\documentclass[12pt]{article}
\usepackage{epsfig,amsmath,amsthm,amssymb,graphics,amsfonts,psfrag,color}
\usepackage[letterpaper,margin=1in]{geometry}
\usepackage[all]{xy}

\theoremstyle{plain}
\newtheorem{thm}{Theorem}[section]
\newtheorem{cor}[thm]{Corollary}
\newtheorem{prop}[thm]{Proposition}
\newtheorem{lem}[thm]{Lemma}

\theoremstyle{definition}

\newtheoremstyle{myremark}
  {5pt}
  {5pt}
  {\sffamily}
  {10pt}
  {\sffamily}
  {:}
  {.5em}
  {}

\theoremstyle{myremark}
\newtheorem*{remark}{Remark}

\def\R{\mathbb{R}}

\newcommand{\be}{\begin{equation}}
\newcommand{\ee}{\end{equation}}
\newcommand{\bea}{\begin{eqnarray}}
\newcommand{\eea}{\end{eqnarray}}
\newcommand{\beann}{\begin{eqnarray*}}
\newcommand{\eeann}{\end{eqnarray*}}
\newcommand{\benn}{\begin{equation*}}
\newcommand{\eenn}{\end{equation*}}

\def\ra{\rightarrow}
\def\I{\infty}
\def\I{\infty}

\newcommand{\cB}{{\mathcal B}}  
\newcommand{\cC}{{\mathcal C}}  
\newcommand{\cD}{{\mathcal D}}  
\newcommand{\cE}{{\mathcal E}}  
\newcommand{\cL}{{\mathcal L}}  
\newcommand{\cM}{{\mathcal M}}  
\newcommand{\cN}{{\mathcal N}}  
\newcommand{\cO}{{\mathcal O}}  
\newcommand{\cS}{{\mathcal S}}  
\newcommand{\cU}{{\mathcal U}}  
\newcommand{\cV}{{\mathcal V}}  

\begin{document}

\author{Christian Kuehn\footnotemark[1]~ and 
Peter Szmolyan\footnotemark[1]}

\renewcommand{\thefootnote}{\fnsymbol{footnote}}
\footnotetext[1]{%
Institute for Analysis and Scientific Computing, 
Vienna University of Technology, 
Vienna, 1040, Austria.} 
\renewcommand{\thefootnote}{\arabic{footnote}}
 
\title{Multiscale Geometry of the Olsen Model and\\ Non-Classical Relaxation Oscillations}

\maketitle

\begin{abstract}
We study the Olsen model for the peroxidase-oxidase reaction. The dynamics is
analyzed using a geometric decomposition based upon multiple time scales. The
Olsen model is four-dimensional, not in a standard form required by geometric
singular perturbation theory and contains multiple small parameters. These three 
obstacles are the main challenges we resolve by our analysis. Scaling and the
blow-up method are used to identify several subsystems. The results presented
here provide a rigorous analysis for two oscillatory modes. In particular, we 
prove the existence of non-classical relaxation oscillations in two cases. The analysis
is based upon desingularization of lines of transcritical and submanifolds
of fold singularities in combination with an integrable relaxation phase. 
In this context our analysis also explains an assumption that has been utilized, 
based purely on numerical reasoning, in a previous bifurcation analysis by 
Desroches, Krauskopf and Osinga 
[{Discret.}~{Contin.}~{Dyn.}~{Syst.}~S, 2(4), p.807--827, 2009]. Furthermore, 
the geometric decomposition we develop forms the basis to prove 
the existence of mixed-mode and chaotic oscillations in the Olsen model, which
will be discussed in more detail in future work. 
\end{abstract}



\section{Introduction \& Review}  
\label{sec:intro}

Experimental observation of oscillatory dynamics \cite{OlsenDegn} in the 
peroxidase-oxidase (PO) reaction
\be
\label{eq:reaction}
2~NADH+2~H^++O_2\ra2~NAD^+ + 2~H_2O
\ee
led to further interest in the dynamical mechanisms \cite{DegnOlsenPerram}. 
Various models have been proposed \cite{OlsonWilliksenScheeline,AgudaLarter,LarterHemkin}
to capture the dynamics of \eqref{eq:reaction}. We are going to study a model
for the PO reaction initially proposed by Degn, Olsen and Perram
\cite{DegnOlsenPerram} (DOP). The four ordinary differential equations (ODEs),
as considered by Olsen \cite{Olsen}, are 
\be
\label{eq:Olsen}
\begin{array}{lcl}
\frac{dA}{dT}&=& -k_3ABY+k_7-k_{-7}A,\\
\frac{dB}{dT}&=& -k_3ABY-k_1BX+k_8,\\
\frac{dX}{dT}&=& k_1BX-2k_2X^2+3k_3ABY-k_4X+k_6,\\
\frac{dY}{dT}&=& -k_3ABY+2k_2X^2-k_5Y,\\
\end{array}
\ee 
where $(A,B,X,Y)\in(\R^4)^+_0=\{(A,B,X,Y)\in\R^4:A\geq 0,B\geq 0,X\geq 0,Y\geq 0\}$
are chemical concentrations and $k_i>0$ are parameters. $A$ and $B$ denote
concentrations of the substrates $NADH$ and $O_2$ while $X$ and $Y$ are
concentrations for two free radicals. We refer to \eqref{eq:Olsen} as the
Olsen model.\medskip 

We briefly describe numerical integration results for the standard parameter
values \cite{Olsen}; see Table \ref{tab:tab1}. Olsen used $k_1$ as a bifurcation
parameter and found three main distinct regimes consisting of mixed-mode oscillations
(MMOs), chaos and relaxation-type periodic oscillations; see Figure \ref{fig:fig1}.
We are going to use the values in Table \ref{tab:tab1} as the main reference parameter
set, where $k_1$ can take three different values. In this paper, we are primarily
interested in periodic oscillations, similar to the results shown in Figure
\ref{fig:fig1}(c) for $k_1=0.41$. However, we shall already indicate how this regime
differs from the other two from the geometric singular perturbation theory (GSPT) 
viewpoint. Our results on the reduction and the existence of periodic orbits of
\eqref{eq:Olsen} are stated in Section \ref{sec:tr_res}.\medskip 

\begin{table}[htbp]
\centering
\begin{tabular}{|c|c|c|c|c|c|c|c|c|}
\hline
$k_1$ &  $k_2$ & $k_3$   & $k_4$ & $k_5$   & $k_6$     & $k_7$ & $k_{-7}$ & $k_8$ \\
\hline 
$0.16,0.35,0.41$ & $250$ & $0.035$ & $20$  & $5.35$  & $10^{-5}$ & $0.8$ & $0.1$    & $0.825$\\
\hline
\end{tabular}
\caption{\label{tab:tab1}Standard parameter values for the Olsen model \eqref{eq:Olsen}.}
\end{table}

\begin{figure}[htbp]
\psfrag{A}{$A$}
\psfrag{t}{$T$}
\psfrag{c}{\scriptsize{(c)}}
\psfrag{a}{\scriptsize{(a)}}
\psfrag{b}{\scriptsize{(b)}}
	\centering
		\includegraphics[width=1\textwidth]{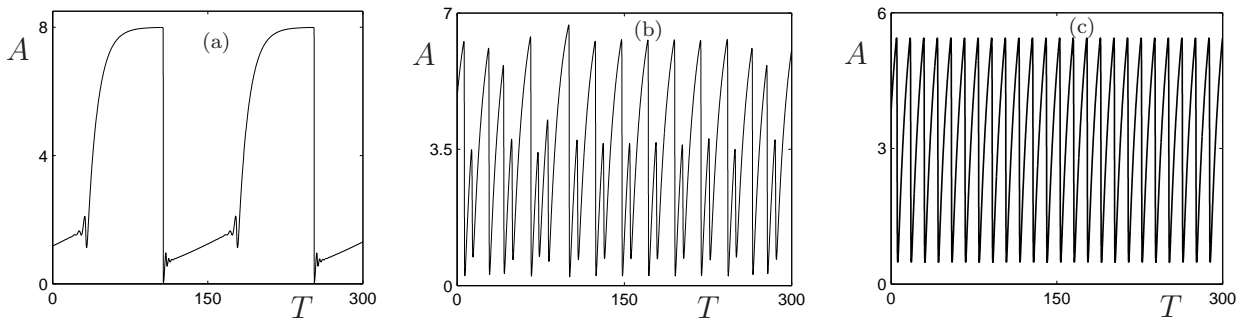}
	\caption{\label{fig:fig1}Numerical simulation for \eqref{eq:Olsen} with 
	parameter values in Table \ref{tab:tab1} upon varying $k_1$. (a) MMOs 
	for $k_1=0.16$, (b) chaotic/aperiodic oscillations for $k_1=0.35$ and 
	(c) regular periodic oscillations for $k_1=0.41$.}
\end{figure}

We briefly review previous work on the peroxidase-oxidase reaction as well as
the mathematical techniques we use for our analysis. An important starting point
are the numerical simulations by Olsen \cite{Olsen} showing that \eqref{eq:Olsen}
can exhibit various types of oscillations depending on the choice of parameters.
Bifurcations of equilibria and sequences of MMOs are investigated in \cite{LarterBushLonisAguda}.
Considering the chemical reaction mechanisms, it was already realized by Aguda, Larter
and Clarke \cite{AgudaLarterClarke} - based on chemical considerations - that the Olsen model
can probably be best understood by decomposition into smaller subsystems. Then various
routes to chaos were proposed ranging from torus break-up \cite{LarterSteinmetz,SteinmetzLarter},
more detailed reaction mechanisms \cite{AgudaLarter} to classical period-doubling scenarios
\cite{SteinmetzGeestLarter}. Subsequently, MMOs, period-doubling sequences and chaotic dynamics
were observed in experiments \cite{HauckSchneider,HauckSchneider1,GeestSteinmetzLarterOlsen}.
Analysis of Lyapunov exponents and period-doubling bifurcations in experimental and
numerical simulation time series provided very strong evidence that chaotic dynamics
occurs \cite{GeestOlsenSteinmetzLarterSchaffer,SteinmetzGeestLarter}. Thompson and Larter
realized the crucial role of multiple time scales in the Olsen model and suggested that
a fast-slow variable decomposition is important to understand the oscillations
\cite{ThompsonLarter}. Multiple time scale structures were also investigated in more
detailed models of the PO reaction. For example, it was conjectured in \cite{HauserOlsen}
that slow manifolds and homoclinic orbits play an important role; Hopf bifurcations 
\cite{BronnikovaFedkinaSchafferOlsen,HauserOlsenBronnikovaSchaffer} and bursting
oscillations were found as well \cite{SchafferBronnikovaOlsen,BronnikovaSchafferOlsen} using
numerical simulation. Recently, a detailed numerical continuation parameter study was carried
out by Desroches, Krauskopf and Osinga \cite{DesrochesKrauskopfOsinga1} who computed various
patterns of periodic orbits, MMOs, chaotic dynamics as well as slow manifolds.\medskip 

Although the reaction-kinetics and the detail of modelling required are still not quite
clear \cite{Scheelineetal} it is evident from the results on the Olsen model that
oscillations, multiple time scales and decomposition are key aspects. However, the
Olsen model has resisted rigorous mathematical analysis for over thirty years, despite it
being a key motivating example to study multiple time scale dynamics \cite{Milik}. In this paper, we provide 
a first detailed GSPT analysis of the Olsen model to understand the geometry of oscillations. 
In fact, our analysis also explains why one has not been able to carry out rigorous geometric
dissection previously. We establish existence results for several special
types of periodic solutions. Furthermore, our analysis provides the basis to prove the existence
of chaos generating mechanisms, which will be detailed elsewhere.\medskip 

The analysis of periodic orbits in multiple time scales systems has an interesting history. 
The seminal work of van der Pol on relaxation oscillations \cite{vanderPol,vanderPol1} is one
of the main starting points for the interest in fast-slow oscillatory systems. The discovery
of ``chaotic relaxation oscillations'' \cite{CartwrightLittlewood,CartwrightLittlewood1} as
well as canard periodic orbits \cite{BenoitCallotDienerDiener,Diener} showed that highly
complicated dynamics can be obtained from rather simple polynomial fast-slow vector fields.
Further analysis revealed that multiple time scale mechanisms can also account for oscillations
with special patterns such as MMOs \cite{Barkley} and bursting \cite{Rinzel}. For recent
reviews on these topics see \cite{Desrochesetal,Izhikevich}. It is important to note that often
the analysis has been carried out in systems with suitably minimal dimension, where local
normal form theory applies, which have a global separation of time scales and which exhibit
a return mechanism similar to the original van der Pol system via S-shaped critical manifolds
\cite{Grasman,MKKR,KuehnRetMaps}. Although several exceptions of this framework have been
considered, for example the general analysis of folded singularities \cite{Wechselberger1},
fast-slow systems in non-standard form \cite{KosiukSzmolyan}, systems with three time scales
\cite{KrupaPopovicKopell} and higher-dimensional systems
\cite{KrupaVidalDesrochesClement,Harveyetal} arising in applications, we are still quite far
away from understanding high-dimensional multiple time scale dynamics in general systems. The
main obstacle for the analysis of the Olsen model is that all problems occur
simultaneously. It is four-dimensional, in nonstandard form, contains several non-folded
degenerate singularities, has three natural small parameters and a return mechanism without an 
S-shaped manifold. It even has multiple regimes of different geometric multiple time scale
decompositions due to the relative asymptotic limits of the small parameters. In this paper, we 
address most of these issues, which are then used to prove the existence of certain periodic 
solutions.\medskip   

The main tools we use to analyze the Olsen model are GSPT, desingularization via the blow-up
method and bifurcation theory in combination with standard techniques from dynamical systems
such as local center manifolds and stability theory. Geometric theory for normally hyperbolic
fast-slow systems was initially developed by Tikhonov~\cite{Tikhonov}, Fenichel~\cite{Fenichel4}
and various other groups~\cite{HirschPughShub}; for recent reviews see~\cite{Jones,Kaper}. For a brief
statement of the main result see Appendix~\ref{ap:fastslow}. 

The blow-up technique was introduced into fast-slow systems by the seminal work of Dumortier and 
Roussarie~\cite{DumortierRoussarie}. It has been used to analyze various local singularities 
such as fold points~\cite{KruSzm3}, folded nodes~\cite{SzmolyanWechselberger1}, Bogdanov-Takens points~\cite{Chiba1}, 
intersection points of slow dynamics~\cite{KruSzm4} and many others. It can also be used to
help to resolve global phenomena such as canard explosion~\cite{KruSzm2}, 
periodic orbits~\cite{GucwaSzmolyan} and homoclinic orbits~\cite{HuberSzmolyan}. Usually blow-up is used for
a distinguished small parameter $\epsilon$, but see~\cite{GucwaSzmolyan2}. Appendix~\ref{ap:blowup}
provides a brief review of the blow-up method.\medskip

The paper is structured as follows: In Section \ref{sec:tr_res} we use a rescaling to describe
a version of the Olsen model which is the starting point of our analysis. Furthermore, we state
our main results. Section \ref{sec:tr_res} ends with a geometric outline for the analysis to 
follow. In Section \ref{sec:main_bu} we employ the blow-up method to desingularize a submanifold of
fold singularities at which the transition between slow drift dynamics and fast large
loops takes place. Section \ref{sec:delay} is dedicated to a much finer analysis of the slow
drift dynamics in a scaling chart of the first blow-up while Section \ref{sec:loops} provides
the study of the fast large loops. In Section \ref{sec:candidate} we construct two classes of 
candidate (or singular limit) trajectories for certain open sets of parameters. In Section 
\ref{sec:retmap} all the previous results are combined to obtain the existence of two types of 
non-classical relaxation oscillations in the Olsen model. An outlook to other oscillatory 
patterns and chaotic dynamics, and their analysis via GSPT, is provided in Section 
\ref{sec:outlook}.

\section{Transformations and the Main Result}
\label{sec:tr_res}

The first step is to scale \eqref{eq:Olsen} to get a better understanding of the multiple 
time scale structure. We use a slight modification of a scaling suggested by Milik \cite{Milik}
\benn
A=\frac{k_1k_5}{k_3\sqrt{2k_2k_8}}a_2,\quad B=\frac{\sqrt{2k_2k_8}}{k_1}b_2,
\quad X=\frac{k_8}{2k_2}x_2,\quad Y=\frac{k_8}{k_5}y_2,\quad 
T=\frac{k_1k_5}{k_3k_8\sqrt{2k_2k_8}}s,
\eenn
which transforms the Olsen model into
\be
\label{eq:Olsen1}
\begin{array}{rcl}
\frac{da_2}{ds}&=& \mu-\alpha a_2 -a_2b_2y_2,\\
\frac{db_2}{ds}&=& \epsilon_b(1-b_2x_2-a_2b_2y_2),\\
\epsilon^2 \frac{dx_2}{ds}&=& b_2x_2-x_2^2 +3a_2b_2y_2-\xi x_2+\delta,\\
\epsilon^2 \frac{dy_2}{ds}&=&\kappa(x_2^2-y_2-a_2b_2y_2),\\
\end{array}
\ee
where $(a_2,b_2,x_2,y_2)\in(\R^4)^+_0$ and the new parameters are given by
\bea
&\mu=\frac{k_7}{k_8},\qquad \alpha=\frac{k_1k_5k_{-7}}{k_3k_8\sqrt{2k_2k_8}},
\qquad \epsilon_b=\frac{k_1^2k_5}{2k_2k_3k_8},
\qquad \kappa= \frac{\sqrt{2k_2k_8}}{k_5},\nonumber\\
&\epsilon^2=\frac{k_3k_8}{k_1k_5},\qquad \xi=\frac{k_4}{\sqrt{2k_2k_8}},\qquad 
\delta=\frac{k_6}{k_8}\label{eq:Milik_para}.
\eea
The reasoning for the choice of subscript for the phase space variables will become
apparent from the blow-up in Section \ref{sec:main_bu}.\medskip

\begin{table}[htbp]
\centering
\begin{tabular}{|c|c|c|c|c|c|c|c|}
\hline
           & $\mu$  &  $\alpha$ & $\epsilon_b$ & $\epsilon^2$ & $\xi$   & $\delta$           & $\kappa$ \\
\hline 
$k_1=0.16$ & $0.97$ & $0.15$    & $0.0095$     & $0.033$      & $0.98$  & $1.2\cdot 10^{-5}$ & $3.93$ \\
$k_1=0.35$ & $0.97$ & $0.32$    & $0.045$      & $0.015$      & $0.98$  & $1.2\cdot 10^{-5}$ & $3.93$ \\
$k_1=0.41$ & $0.97$ & $0.37$    & $0.062$      & $0.013$      & $0.98$  & $1.2\cdot 10^{-5}$ & $3.93$ \\
\hline
\end{tabular}
\caption{\label{tab:tab2}Standard parameter values for the Olsen model \eqref{eq:Olsen1}
obtained via the transformation \eqref{eq:Milik_para} from Table \ref{tab:tab1}; only
approximate values for two significant digits are given.}
\end{table}

The original parameter values by Olsen from Table \ref{tab:tab1} are converted into
the new parameters in Table \ref{tab:tab2}. The transformation already makes the multiple
time scale structure of the Olsen model more visible. It is very important to note from
Table \ref{tab:tab2} that varying $k_1$ changes the orders of magnitude for the small
parameters $\epsilon_b$ and $\epsilon$ as well as their relative size. This effect has to
be used in the mathematical analysis to distinguish different regimes; see also Section 
\ref{sec:outlook}.\medskip

The general strategy to understand the Olsen model, as shown for other multiple time scale systems 
{e.g.}~in \cite{GucwaSzmolyan}, will be to first resolve the fastest dynamics of 
\eqref{eq:Olsen1}. The fastest dynamics is visible using the rescaling
\be
\label{eq:prelim_scale}
a=a_2,\qquad  b=b_2,\qquad x=\epsilon x_2,\qquad y=\epsilon^2 y_2,\qquad \tau=\epsilon^{-2}s   
\ee
which, upon applying \eqref{eq:prelim_scale} to \eqref{eq:Olsen1}, yields
\be
\label{eq:Olsen2}
\begin{array}{rcl}
\frac{da}{d\tau}&=& \epsilon^2(\mu-\alpha a) -aby,\\
\frac{db}{d\tau}&=& \epsilon (\epsilon_b\epsilon-\epsilon_bbx)-\epsilon_baby,\\
\epsilon \frac{dx}{d\tau}&=& -x^2 +\epsilon (b-\xi) x +3aby +\epsilon^2 \delta,\\
\frac{dy}{d\tau}&=&\kappa(x^2-y-aby).\\
\end{array}
\ee
The two systems \eqref{eq:Olsen1} and \eqref{eq:Olsen2} are going to be two main 
components of our analysis. Notice that different regimes can exist depending upon
the (relative) size of the three natural small parameters $\epsilon$, $\epsilon_b$
and $\delta$. In fact, just viewing \eqref{eq:Olsen2} on a formal level, all the
different fast-slow possibilities for a four-dimensional system occur in Olsen model:

\begin{itemize}
 \item for \eqref{eq:Olsen1}, $\epsilon^2\ra 0$ yields two fast and two slow variables, 
 \item for \eqref{eq:Olsen1}, $\epsilon^2\neq 0$ and $\epsilon_b\ra 0$ yields three fast
 variables and one slow variable, 
 \item for \eqref{eq:Olsen2}, $\epsilon\ra 0$ and $\epsilon_b\neq 0$ yields one fast
 variable and three slow variables.
\end{itemize}

All the different regimes are relevant for the asymptotic analysis of oscillatory
dynamics in the Olsen model. In particular, three major regimes are relevant
\benn
\epsilon_b\ll \epsilon^2,\qquad \epsilon_b\approx \epsilon^2,\qquad \epsilon_b\gg \epsilon^2, 
\eenn
which roughly correspond to the three cases $k_1=0.16$, $k_1=0.35$ and $k_1=0.41$ from
Table \ref{tab:tab2}. In this paper, we focus on regular oscillations as displayed in Figure 
\ref{fig:fig1}(c); but see Section \ref{sec:outlook} for the other two regimes. For the 
analysis in Sections \ref{sec:main_bu}-\ref{sec:retmap} we assume that
\be
\label{eq:asymp_main_assume}
0< \epsilon^2\ll \epsilon_b,
\ee
where $\epsilon_b$ will be regarded as a fixed parameter and singular limits are only
considered with respect to a single time scale separation parameter $0< \epsilon\ll 1$.
Then observe that \eqref{eq:Olsen2} has a critical manifold for the singular limit 
$\epsilon=0$ given by
\be
\label{eq:main_attract}
\cC_0=\left\{ (x,y,a,b)\in \R^4:\frac{x^2}{3ab}=y\right\}. 
\ee
For all the oscillatory patterns we are interested in, the conditions $a>a^*\geq 0$ and $b>b^*\geq 0$
hold for suitable bounded constants $a^*$ and $b^*$. This implies that $\cC_0$ is a well-defined
critical manifold in the region 
\benn
\cD:=\{(a,b,x,y)\in \R^4:a>a^*,b>b^*,x\geq 0, y\geq 0\}.
\eenn 
We are going to assume from now on that all calculations are carried out within $\cD$.
Hence, all sets in the following are understood as intersections with $\cD$. Then
$\cC_0$ is normally hyperbolic attracting for \eqref{eq:Olsen2} when $x>0$. Indeed, 
consider the fastest component of the vector field on the time scale 
$\tilde{\tau}:=\tau/\epsilon$ given by
\benn
F(a,b,x,y;\epsilon):=-x^2 +\epsilon (b-\xi) x +3aby +\epsilon^2 \delta.
\eenn
Then the attraction for $x>0$ follows since we just have 
\benn
\left.\left[\frac{\partial F}{\partial x}(a,b,x,y;0)\right]\right|_{\{x>0\}}
=\left.-2x\right|_{\{x>0\}}<0.
\eenn
The results of Fenichel \cite{Fenichel4} and Tikhonov \cite{Tikhonov} (see Appendix 
\ref{ap:fastslow}) immediately imply the next result.

\begin{prop}
\label{prop:Fenichel_prop_C0}
Consider a trajectory 
\benn
\gamma(\tau)=(a(\tau),b(\tau),x(\tau),y(\tau)),\qquad \tau\in[0,T], ~T>0
\eenn
of \eqref{eq:Olsen2} with initial value $a(0),b(0),x(0),y(0)\in\cD$ such that
\benn
a(0),b(0),x(0),y(0)>0\quad\text{and}\quad a(0),b(0),x(0),y(0)=\cO(1)\text{ as $\epsilon\ra0$.}
\eenn
Assume $0<\epsilon\ll1$ is sufficiently small and all other parameters
are fixed and positive. Then $\gamma(\tau)$ is $\cO(e^{-K/\epsilon})$-close 
to the slow manifold $\cC_\epsilon$ after a finite time $\tau^*\in[0,T]$.
\end{prop}

Proposition \eqref{prop:Fenichel_prop_C0} essentially describes the fastest initial dynamics
for most initial conditions. Trajectories are just attracted towards $\cC_0$. The three-dimensional
flow on $\cC_0$ in the normally hyperbolic regime is considered in Section \ref{sec:loops}. It 
will be shown that trajectories may also reach a neighbourhood of the set 
\benn
\cL_0:=\{(a,b,x,y)\in\cD:x=0=y\}\subset \cC_0.
\eenn
We observe that $\cL_0$ is a submanifold  of non-degenerate fold points since 
\benn
F(a,b,0,0;0)=0,\quad \frac{\partial F}{\partial x}(a,b,0,0;0)=0,\quad 
\frac{\partial^2 F}{\partial x^2}(a,b,0,0;0)=0,\quad
\frac{\partial F}{\partial y}(a,b,0,0;0)\neq0, 
\eenn
where we used that $a>a^*>0$, $b>b^*>0$ in $\cD$ for the $y$-derivative. The fold manifold
$\cL_0$ is not normally hyperbolic and has to be desingularized. The analysis of the fold
region is contained in Sections \ref{sec:main_bu}-\ref{sec:delay}.\medskip

\begin{figure}[htbp]
\psfrag{af}{\scriptsize{(a)}}
\psfrag{bf}{\scriptsize{(b)}}
\psfrag{ab}{\scriptsize{$a,b$}}
\psfrag{a}{\scriptsize{$a_2$}}
\psfrag{b}{\scriptsize{$b_2$}}
\psfrag{x}{\scriptsize{$x$}}
\psfrag{y}{\scriptsize{$y$}}
\psfrag{xy}{\scriptsize{$x_2,y_2$}}
\psfrag{L0}{\scriptsize{$\cL_0$}}
\psfrag{C0}{\scriptsize{$C_0$}}
\psfrag{gc}{\scriptsize{$\gamma_c$}}
\psfrag{gj}{\scriptsize{$\gamma_j$}}
\psfrag{xy0}{\scriptsize{$\{x_2=0=y_2\}$}}
\psfrag{bxi}{\scriptsize{$b_2=\xi$}}
\psfrag{2ab1}{\scriptsize{$2a_2b_2=1$}}
	\centering
		\includegraphics[width=1\textwidth]{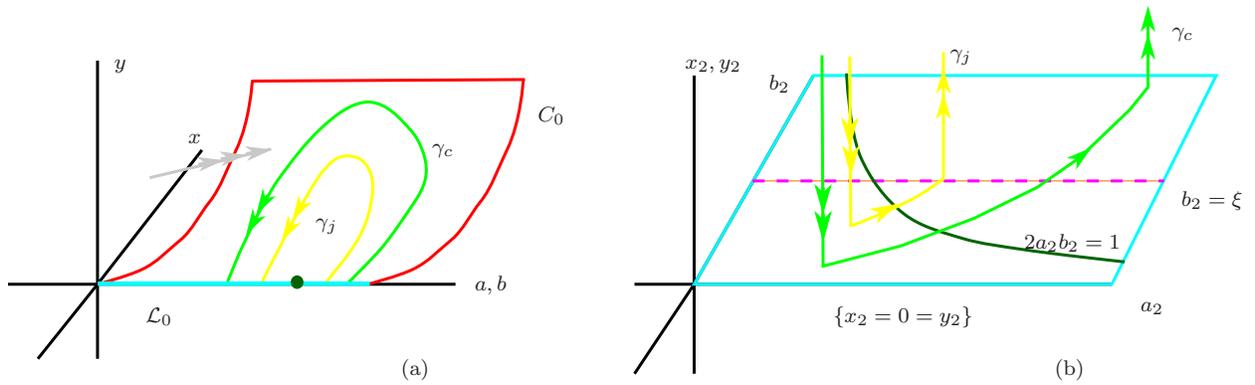}
	\caption{\label{fig:fig8}Sketch of the basic geometry for the two types of 
	non-classical relaxation oscillations inside the region $\cD$. (a) Phase space for
	the system \eqref{eq:Olsen2} which captures the large fast loops. The critical 
	manifold $C_0$ (red), two segments of the candidate orbits $\gamma_c$ (green) and $\gamma_j$ (yellow),
	the fold manifold $\cL_0$ (cyan), the submanifold $\{2ab=1,x=0=y\}$ (dark-green dot) and
	the ``super-fast'' attracting dynamics (grey triple arrow) are shown. (b) Phase space
	for \eqref{eq:Olsen2} which focuses on the slow drift near $\cL_0$ (cyan). We show
	segments of the two candidate orbits $\gamma_c$ (green), $\gamma_j$ (yellow), 
	the exchange-of-stability line $\{b_2=\xi\}$ (magenta) and the submanifold 
	$\{2a_2b_2=1,x_2=0=y_2\}$ (dark-green curve). For a description of the dynamics 
	please refer to the text in Section \ref{sec:tr_res}.}
\end{figure}

Before we proceed to state our main result, we shall motivate the geometric construction
briefly on a non-rigorous level as outlined in Figure \ref{fig:fig8}. For the following discussion, 
we refer to objects in singular limits, which we have to perturb later on. We 
start with system \eqref{eq:Olsen2} and apply Proposition \ref{prop:Fenichel_prop_C0} to 
understand the ``super-fast'' dynamics on the time scale $\tilde{\tau}:=\tau/\epsilon$. 
Trajectories get attracted to $\cC_0$. On $\cC_0$, \eqref{eq:Olsen2} yields 
a three-dimensional vector field on the time scale $\tau$
\be
\label{eq:sf_intro_outer}
\begin{array}{rcl}
\frac{da}{d\tau}&=& -aby,\\
\frac{db}{d\tau}&=& -\epsilon_baby,\\
\frac{dy}{d\tau}&=&\kappa(2ab-1) y,\\
\end{array}
\ee
where we have just used $x^2=3aby$ and $\epsilon=0$. It turns out that \eqref{eq:sf_intro_outer}
can be solved explicitly, albeit with relatively inconvenient formulas. Although 
\eqref{eq:sf_intro_outer} is formally a ``slow subsystem'' we shall refer to it as the fast
dynamics as we shall discover another (even ``slower'') system inside $\cL_0$. After some 
calculations, the solutions of \eqref{eq:sf_intro_outer} turn out to be arcs as indicated by 
Figure \ref{fig:fig8}(a) connecting two points on $\cL_0$. Furthermore, one can view
these solutions as jumps over a submanifold $\{2ab=1\}$, which we indicated as a dot in 
Figure \ref{fig:fig8}(a). Since these arcs start and end in the singular locus of 
fold points $\cL_0$ we proceed to system \eqref{eq:Olsen1}, which is a ``zoom'' of 
\eqref{eq:Olsen2} near $\cL_0$. One notices that for \eqref{eq:Olsen1}, upon taking 
$\epsilon=0=\delta$, one part of the two-dimensional critical manifold is 
given by $\{x_2=0=y_2\}$. The results from Section \ref{sec:main_bu} are going to yield 
that $\{x_2=0=y_2\}$ is attracting for $b_2<\xi$ and repelling for $b_2>\xi$; see also 
Figure \ref{fig:fig8}(b). We denote these attracting and repelling parts 
by $\cS_{2,0}^{a-}$ and $\cS_{2,0}^{r+}$. Let us follow candidate trajectories which 
get attracted to $\cS_{2,0}^{a-}$, such as $\gamma_c$ or $\gamma_j$ shown in Figure 
\ref{fig:fig8}(b). Once $\gamma_{c,j}$ reach $\cS_{2,0}^{a-}$ their dynamics is governed by 
taking $\epsilon=0$ in \eqref{eq:Olsen2} {i.e.}
\be
\label{eq:sf_intro_inner}
\begin{array}{rcl}
\frac{da_2}{ds}&=& \mu-\alpha a_2 ,\\
\frac{db_2}{ds}&=& \epsilon_b,\\
\end{array}
\ee 
which has very simple explicit solution formulas. However, the fast direction stability 
changes at $b_2=\xi$. It can be proven that for $\delta=0$, we may view $\{b_2=\xi\}$ as
a submanifold of transcritical singularities where maximal delay occurs so that $\gamma_c$
is a canard trajectory traveling for a considerable time onto $\cS_{2,0}^{r+}$ before it 
eventually jumps; see Figure \ref{fig:fig8}(b). However, if $\delta>0$ is positive and not 
exponentially small with respect to $\epsilon$, then we are in the case $\gamma_j$ where 
the candidate orbit jumps near $\{b_2=\xi\}$. In both cases, $\gamma_{c,j}$ are then in a 
fast regime away after their departure from $\cS_{2,0}^{r+}$, which allows us to connect them back from Figure 
\ref{fig:fig8}(b) to \ref{fig:fig8}(a). Taking a global view, it is then possible to 
construct two types of candidate periodic orbits $\gamma_j$ and $\gamma_c$ which can 
then be shown to perturb to periodic orbits for $0<\epsilon\ll1$. The precise statement 
is as follows:

\begin{thm}
\label{thm:main_intro}
There exists a family of open sets $(\mu_{1}(\epsilon),\mu_{2}(\epsilon))$ for some 
$\mu_i>0$ with $i=1,2$ and $\epsilon_0>0$ sufficiently small such that the Olsen 
model \eqref{eq:Olsen2} for $\mu\in(\mu_1(\epsilon),\mu_2(\epsilon))$ with 
$\epsilon\in(0,\epsilon_0]$ and otherwise standard parameter values from 
Table \ref{tab:tab2} with $k_1=0.41$ has a family of periodic orbits $\psi_{\epsilon}$
in the following two cases:
\begin{enumerate}
 \item \textbf{Canard case}: Suppose $\delta=\cO(\epsilon^2e^{-K_1/\epsilon^2})$ and $K_1>0$ is 
 some fixed constant independent of $\epsilon$. Then $\psi_\epsilon$ has 
 a canard segment which is $\cO(\epsilon^2)$-close to a repelling part of $\{x_2=0=y_2\}$
 for a time $s^*=\cO(1)$, $s^*>0$ as $\epsilon\ra 0$.
 \item \textbf{Jump case:} Suppose $\delta=K_2\epsilon^2$, 
 $K_2>0$ and $K_2$ is fixed as $\epsilon\ra 0$. Then 
 $\psi_\epsilon$ does not have a canard segment and leaves $\cL_0$ in an 
 $\epsilon$-dependent neighbourhood $\cN(\epsilon)$ of $\{b_2=\xi\}$ such that 
 $d_{\textnormal{H}}(\cN(\epsilon),\{b_2=\xi\})\ra 0$ as $\epsilon\ra 0$. 
\end{enumerate}
In both cases, $\psi_{0}$ is a candidate orbit with a slow segment in $\cL_0$ 
and a fast segment in $\cC_0$ and $d_{\textnormal{H}}(\psi_\epsilon,\psi_0)\ra 0$ 
as $\epsilon\ra 0$. In both cases, $\psi_\epsilon$ is locally asymptotically stable.
\end{thm} 

The situation is also illustrated in Figure \ref{fig:fig8}. Heuristically, in 
view of the discussion preceding Theorem \ref{thm:main_intro},
we may concisely summarize the result as follows (see Figure \ref{fig:fig8}): 
\begin{enumerate}
\item If $\delta$ is zero or exponentially small then we have a periodic orbit
converging to a candidate orbit with a canard segment 
{i.e.}~$d_{\textnormal{H}}(\psi_\epsilon,\gamma_c)\ra 0$ as $\epsilon\ra 0$.
\item If $\delta$ scales like $\epsilon^2$ then we have a periodic orbit
converging to a candidate orbit, without a canard segment and jumping near
a (transcritical) singularity,
{i.e.}~$d_{\textnormal{H}}(\psi_\epsilon,\gamma_j)\ra 0$ as $\epsilon\ra 0$.
\end{enumerate}

Of course, one may also aim to consider other situations. For example, if 
$\delta\gg \epsilon^2$ then our analysis does not apply but this case does not occur 
in the original parameter sets used by Olsen, so we shall not discuss it here.
However, there is an interesting case which occurs when
\be
\label{eq:delta_deform}
0<\epsilon^2e^{-K/\epsilon}\ll \delta\ll \epsilon^2\ll 1.
\ee
In this case, the family periodic orbits with a canard segment deforms smoothly
into the periodic family of the jump case as $\delta$ increases. In fact, this
is precisely the case which occurs for the classical Olsen parameter values 
from Table \ref{tab:tab2} since $\epsilon^2=1.5\cdot10^{-2}$ and 
$\delta=1.2\cdot 10^{-5}$. As usual when applying GSPT, it is helpful to focus 
on the two limiting cases to describe an intermediate asymptotic regime. Desroches 
et {al.}~\cite{DesrochesKrauskopfOsinga1} computed numerical bifurcation diagrams
for the Olsen model and observed that ``the bifurcation structure does not change
in an essential way'' \cite{DesrochesKrauskopfOsinga1} when the same types of diagrams
were computed for $\delta=0$ and $\delta=1.2\times 10^{-5}$. In fact, our result shows
that there will be a substantial deformation of orbits in the system depending upon
$\delta$. However, this is no contradiction as the bifurcation diagram may not change
significantly, when plotted in parameter space only, as there is a family
of periodic orbits, whether $\delta=0$ or $\delta=1.2\cdot 10^{-5}$.\medskip

The orbit $\psi_\epsilon$ from Theorem \ref{thm:main_intro} has relaxation-type 
properties as it consists of alternating fast and slow segments but it is not a 
classical relaxation oscillation generated by a cubic critical manifold mechanism 
\cite{vanderPol1,Grasman}. Hence we use the term non-classical relaxation 
oscillation. In Sections \ref{sec:main_bu}-\ref{sec:retmap} we proceed to provide 
a proof of Theorem \ref{thm:main_intro}.

\section{The Main Blow-Up}
\label{sec:main_bu}

We start with the analysis near the fold locus $\cL_0$ which will require a blow-up; 
see Appendix \ref{ap:blowup} as well as \cite{Dumortier1,KruSzm1} for background on
geometric desingularization via the blow-up method. Coefficients to desingularize
\eqref{eq:Olsen2} are suggested by the scaling \eqref{eq:prelim_scale}. Let 
\benn
\bar{\cD}:=[a^*,\I)\times [b^*,\I) \times (\cS^2)^+_0\times [0,r_0] 
\eenn
for $r_0>0$ where $(\cS^2)^+_0\subset \R^3$ denotes the upper half of the unit sphere
including the equator. Changing the time scale to $t=\tau/\epsilon$ and augmenting
\eqref{eq:Olsen2} by $\epsilon'=0$ yields 
\be
\label{eq:Olsen2a}
\begin{array}{rcl}
a'&=& \epsilon^3(\mu-\alpha a) -\epsilon aby,\\
b'&=& \epsilon^2 (\epsilon_b\epsilon-\epsilon_bbx)-\epsilon\epsilon_baby,\\
x'&=& -x^2 +\epsilon (b-\xi) x +3aby +\epsilon^2 \delta,\\
y'&=&\epsilon\kappa(x^2-y-aby),\\
\epsilon'&=&0.
\end{array}
\ee
Consider the blow-up transformation $\Phi:\bar{\cD}\ra \cD$ defined via
\be
\label{eq:blowup1}
a=\bar{a},\qquad b=\bar{b},\qquad x=\bar{r} \bar{x},\qquad
y=\bar{r}^2 \bar{y},\qquad \epsilon=\bar{r} \bar{\epsilon} 
\ee
where $(\bar{x},\bar{y},\bar{\epsilon})\in (\cS^2)^+_0$. $\Phi$ blows up the
vector field $V$ given by \eqref{eq:Olsen2a}; see also Figure \ref{fig:fig7}(b). 
The map $\Phi$ induces a vector field
$\bar{V}$ on $\bar{\cD}$ by pushforward $\Phi_*(\bar{V})=V$. To analyze $\bar{V}$
it is convenient to consider the manifold $\bar{\cD}$ in several charts. Define
the following submanifolds 
\benn
\bar{\cD}_{\bar{x}}:=\bar{\cD}\cap\{\bar{x}>0\}\qquad 
\text{and}\qquad \bar{\cD}_{\bar{\epsilon}} :=\bar{\cD}\cap\{\bar{\epsilon}>0\}.
\eenn
The submanifold $\bar{\cD}_{\bar{y}}$ can also be considered but will yield the 
same qualitative view of the dynamics as $\bar{\cD}_{\bar{x}}$. Hence we are not 
going to need it for our analysis.

\begin{figure}[htbp]
\psfrag{y}{$y_2$}
\psfrag{a}{$a_2$}
\psfrag{x}{$x_2$}
\psfrag{ybar}{$\bar{y}$}
\psfrag{abar}{$\bar{a}$}
\psfrag{xbar}{$\bar{x}$}
\psfrag{al}{\tiny{$a_2$}}
\psfrag{s}{\tiny{$s$}}
\psfrag{alabel1}{\scriptsize{(a1)}}
\psfrag{alabel2}{\scriptsize{(a2)}}
\psfrag{blabel1}{\scriptsize{(b1)}}
\psfrag{blabel2}{\scriptsize{(b2)}}
\psfrag{psieps}{$\psi_\epsilon$}
\psfrag{C0}{$\cC_0$}
\psfrag{C0bar}{$\bar{\cC}_0$}
\psfrag{cylinder}{$r=0$}
	\centering
		\includegraphics[width=1\textwidth]{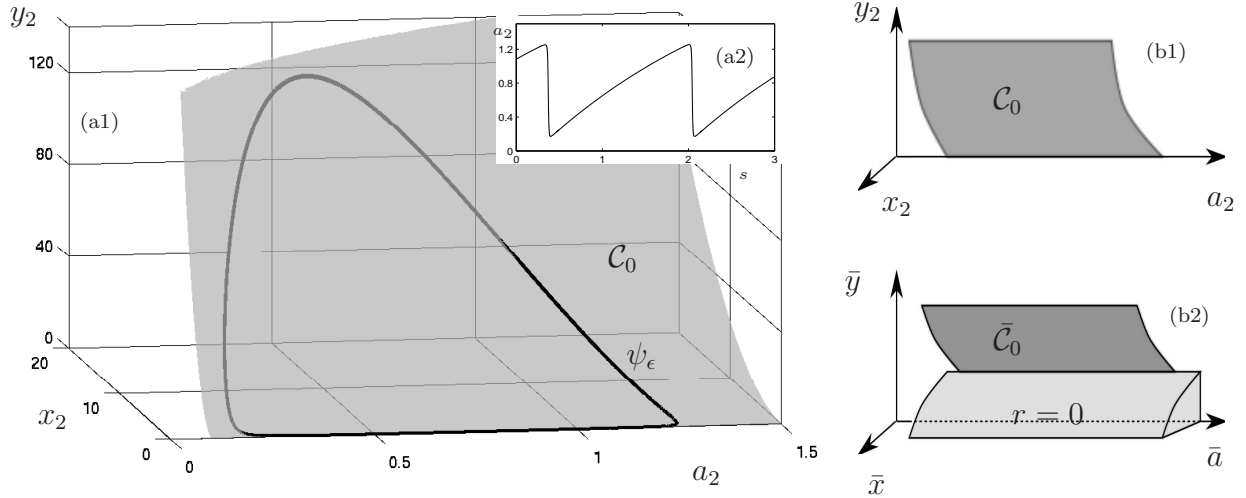}
	\caption{\label{fig:fig7}Illustration of the non-classical relaxation 
	orbit $\psi_\epsilon$, the critical manifold $\cC_0$ and the blow-up. 
	(a1) Projection into $(a_2,x_2,y_2)$-space of an integration of the full 
	system \eqref{eq:Olsen1} with standard parameter values from Table 
	\ref{tab:tab2} and $k_1=0.41$. (a2) Time series for the variable $a$. 
	(b1) Sketch of the situation before the blow-up with the critical manifold 
	(dark gray). (b2) Blown-up space where the fold points have been 
	desingularized by the transformation \eqref{eq:blowup1} inserting a 
	cylinder (light gray) giving the new domain $\bar{\cD}$.}
\end{figure}
  
\begin{lem}
\label{lem:blowup_change}
The maps $\kappa_1:\bar{\cD}_{\bar{x}}\ra \cD$ for 
$(a_1,b_1,r_1,y_1,\epsilon_1)\in \bar{\cD}_{\bar{x}}$ and 
$\kappa_2:\bar{\cD}_{\bar{\epsilon}}\ra \cD$ for 
$(a_2,b_2,x_2,y_2,r_2)\in \bar{\cD}_{\bar{\epsilon}}$ given by
\benn
\begin{array}{lllll}
a_1=\bar{a},\quad & b_1=\bar{b}, \quad & r_1=\bar{r}\bar{x},
\quad & y_1=\bar{x}^{-2}\bar{y},\quad & \epsilon_1=\bar{x}^{-1}\bar{\epsilon}\\
a_2=\bar{a},\quad & b_2=\bar{b}, \quad & x_2=\bar{\epsilon}^{-1}\bar{x},
\quad & y_2=\bar{\epsilon}^{-2}\bar{y},\quad & r_2=\bar{r}\bar{\epsilon}\\
\end{array}
\eenn
define charts for $\bar{\cD}$ in which the blow-up $\Phi$ is, respectively, given by
\be
\label{eq:main_blowup_formal}
\begin{array}{lllll}
a=a_1,\quad & b=b_1, \quad & x=r_1,\quad & y=r_1^2y_1,\quad & \epsilon=r_1\epsilon_1,\\
a=a_2,\quad & b=b_2, \quad & x=r_2x_2,\quad & y=r_2^2y_2,\quad & \epsilon=r_2.\\
\end{array}
\ee
\end{lem}

\begin{proof}
Consider $\kappa_1$ then $r_1=\bar{r}\bar{x}$ but since for $\Phi$ we 
have $\bar{r}\bar{x}=x$ it follows that $x=r_1$. Furthermore
\benn
r_1^2y_1=r_1^2\bar{x}^{-2}\bar{y}=x^2\bar{x}^{-2}\bar{y}
=\bar{r}^2\bar{x}^2\bar{x}^{-2}\bar{y}=\bar{r}^2\bar{y}=y.
\eenn
The calculation for $\epsilon$ and the second chart $\kappa_2$ are similar.
\end{proof}

Observe that the blow-ups \eqref{eq:main_blowup_formal} in the charts 
$\kappa_1$ and $\kappa_2$ are essentially defined by the conditions 
$\bar{x}=1$ and $\bar{\epsilon}=1$.
 
\begin{lem}
\label{lem:main_blowup_cchange}
The coordinate change $\kappa_{12}$ from the first to the second 
chart and its inverse are
\be
\label{eq:main_blowup_cchange}
\begin{array}{lllll}
a_2=a_1,\quad & b_2=b_1,\quad & x_2=\epsilon_1^{-1},
\quad & y_2=y_1\epsilon_1^{-2},\quad & r_2=r_1\epsilon_1\\
a_1=a_2,\quad & b_1=b_2,\quad & r_1=r_2x_2,
\quad & y_1=y_2x_2^{-2},\quad & \epsilon_1=x_2^{-1}.\\
\end{array}
\ee
\end{lem}

With the charts available we can calculate the blown-up vector fields in each chart. 

\begin{lem}
\label{lem:chart1_vf}
In the chart $\kappa_1$ the desingularized blown-up vector field is given by
\be
\label{eq:bu_kappa1_vf_lem}
\begin{array}{lcl}
a_1'&=& \epsilon_1r_1^2\left[\epsilon_1^2(\mu-\alpha a_1)-a_1b_1y_1\right],\\
b_1'&=& \epsilon_1r_1^2\epsilon_b\left[\epsilon_1^2-\epsilon_1b_1-a_1b_1y_1\right],\\
r_1'&=& r_1\left[-1+\epsilon_1(b_1-\xi)+3a_1b_1y_1+\epsilon_1^2\delta\right],\\
y_1'&=& \kappa\epsilon_1(1-y_1(1+a_1b_1))-2y_1\left(-1+\epsilon_1(b_1-\xi)
+3a_1b_1y_1+\epsilon_1^2\delta\right),\\
\epsilon_1'&=&-\epsilon_1\left[-1+\epsilon_1(b_1-\xi)+3a_1b_1y_1+\epsilon_1^2\delta\right].\\
\end{array}
\ee
\end{lem}

\begin{proof}
The equations for $a_1$, $b_1$ and $r_1$ are easy to obtain. For $y_1$ we calculate
\benn
y'=2r_1r_1'y_1+r_1^2y_1'\qquad \Rightarrow\quad y_1'=\frac{y'-2r_1y_1r_1'}{r_1^2}.
\eenn 
Substituting $y'$ from \eqref{eq:Olsen2a} and using \eqref{eq:main_blowup_formal} 
gives $y_1'$. The calculation for $\epsilon_1'$ is easier since $\epsilon'=0$. 
All equations derived in this way have a multiplicative pre-factor of $r_1$, which 
can be removed by a time rescaling which yields the desingularized vector field 
\eqref{eq:bu_kappa1_vf_lem}. 
\end{proof}

In the chart $\kappa_2$ the blow-up \eqref{eq:main_blowup_formal} reduces to 
the rescaling 
\be
\label{eq:blowup1_scaling}
a=a_2,\qquad b=b_2,\qquad x=\epsilon x_2,\qquad y=\epsilon^2 y_2. 
\ee

\begin{lem}
\label{lem:chart2_flow}
In the chart $\kappa_2$ the blown-up vector field is given by
\be
\label{eq:Olsen1_bu_res}
\begin{array}{rcl}
\frac{da_2}{ds}&=& \mu-\alpha a_2 -a_2b_2y_2,\\
\frac{db_2}{ds}&=& \epsilon_b(1-b_2x_2 -a_2b_2y_2),\\
\epsilon^2\frac{dx_2}{ds}&=& 3a_2b_2y_2-x_2^2+(b_2-\xi)x_2+\delta,\\
\epsilon^2\frac{dy_2}{ds}&=& \kappa(x_2^2-y_2-a_2b_2y_2).\\
\end{array}
\ee 
where the time scale is $s=\epsilon^2\tau$.
\end{lem}

Hence, the re-scaled version \eqref{eq:Olsen1} of the Olsen model we derived 
from Olsen's original equations \eqref{eq:Olsen} just resolves the dynamics 
well on one scale in a certain region of phase space. When the blow-up reduces
to a re-scaling as in $\kappa_2$, then one also refers to the corresponding chart
as the classical chart \cite{SzmolyanWechselberger1}. The chart $\kappa_1$
describes the regime where trajectories approach the submanifold of folds $\cL_0$ 
from the three-dimensional slow flow discussed in Section \ref{sec:loops}. Hence 
we are going to discuss $\kappa_1$ first.

\subsection{First Chart}

The approach towards and departure from a submanifold 
\benn
[a^*,\I)\times [b^*,\I) \times (\cS^2)^+_0\times  \{\bar{r}=0\}\cap\{\bar{x}>0\}
\eenn
consisting of an $(\bar{a},\bar{b})$-dependent family of spheres can be studied
best in the chart $\kappa_1$. In particular, we study the ODEs 
\eqref{eq:bu_kappa1_vf_lem} from Lemma \ref{lem:chart1_vf} in this section. 
The case $\epsilon_1=0$ corresponds to the equator of the spheres.

\begin{lem}
\label{lem:first_leaves}
There exists a dimension two foliation with leaves 
\be
\label{eq:first_leaves}
\{\epsilon_1=0,a_1=a_1^*,b_1=b_1^*\}
\ee
with constants $a_1^*$, $b_1^*$ for \eqref{eq:bu_kappa1_vf_lem}. The vector field
in the invariant submanifolds \eqref{eq:first_leaves} is given by
\be
\label{eq:bu_kappa1_fol1}
\begin{array}{lcr}
r_1'&=& r_1\left(3a_1^*b_1^*y_1-1\right),\\
y_1'&=& -2y_1\left(3a_1^*b_1^*y_1-1\right).\\
\end{array}
\ee
\end{lem}

The proof of Lemma \ref{lem:first_leaves} follows by direct substitution 
of the algebraic conditions defining \eqref{eq:first_leaves} into 
\eqref{eq:bu_kappa1_vf_lem}. The planar system \eqref{eq:bu_kappa1_fol1} can 
be analyzed directly using standard phase plane methods and linearization. 
Recall that we are only interested in the case $y_1\geq0$ and $r_1\geq 0$.

\begin{lem}
\label{lem:inv_ry}
The ODE \eqref{eq:bu_kappa1_fol1} has (see also Figure \ref{fig:fig4})
\begin{itemize}
 \item a saddle equilibrium at $(r_1,y_1)=(0,0)$ with eigenvalues $\lambda_1=-1$,
 $\lambda_2=2$ and eigendirections $v_1=(1,0)^T$, $v_2=(0,1)^T$;
 \item a line of degenerate equilibrium points $\{y_1=1/(3a_1^*b_1^*)\}$ which is
 attracting in the $v_2$ direction.  
\end{itemize}
The line $\{y_1=1/(3a_1^*b_1^*)\}$ corresponds to the attracting critical
manifold $\cC_0$ defined in \eqref{eq:main_attract}.
\end{lem}

\begin{proof}
The calculations to find the equilibria and their stability are straightforward. 
Regarding the last statement about $\cC_0$, observe that Lemma 
\ref{lem:blowup_change} implies $x^2=r_1^2$ and $y=r_1^2y_1$ so 
\benn
\Phi\circ\kappa_1^{-1}\left(\{y_1=1/(3a_1b_1)\}\right)=
\Phi\circ\kappa_1^{-1}\left(\{y_1r_1^2=r_1^2/(3a_1b_1)\}\right)=\{y=x^2/(3ab)\}. 
\qedhere
\eenn
\end{proof}

There is a natural second family of invariant subspaces for \eqref{eq:bu_kappa1_vf_lem} 
for the case $r_1=0$ (i.e.~``on the sphere'') which yields a more complicated family 
of flows. To analyze this case, we shall assume that 
\be
\label{eq:delta_vanish}
\delta=\delta(\epsilon)\qquad \text{and}\qquad \delta(0)=0.
\ee
Note that \eqref{eq:delta_vanish} holds for the canard case and the jump case 
in Theorem \ref{thm:main_intro}.

\begin{lem}
Suppose \eqref{eq:delta_vanish} holds. Then there exists a dimension two foliation 
with leaves 
\be
\label{eq:second_leaves}
\{r_1=0,a_1=a_1^*,b_1=b_1^*\}
\ee
with constants $a_1^*$, $b_1^*$ for \eqref{eq:bu_kappa1_vf_lem}. 
The vector field in the invariant submanifolds \eqref{eq:second_leaves} is given by
\be
\label{eq:bu_kappa1_fol2}
\begin{array}{lcl}
y_1'&=& \kappa\epsilon_1(1-y_1(1+a^*_1b^*_1))-2y_1\left(-1+\epsilon_1(b^*_1-\xi)
+3a^*_1b^*_1y_1\right),\\
\epsilon_1'&=&-\epsilon_1\left[-1+\epsilon_1(b^*_1-\xi)+3a^*_1b^*_1y_1\right].\\
\end{array}
\ee
\end{lem}

The proof of Lemma \ref{lem:first_leaves} follows by direct substitution 
of the algebraic conditions defining \eqref{eq:second_leaves} into 
\eqref{eq:bu_kappa1_vf_lem} and using $\delta(0)=0$. For the analysis of 
\eqref{eq:bu_kappa1_fol2} we start with the case
\be
\label{eq:as_bxi}
|\xi-b_1^*|\geq K>0\qquad \text{for a fixed constant $K$ 
independent of $\epsilon$.}
\ee
The situation near $|\xi-b_1^*|= 0$ is different and will be covered at 
the end of this section. If \eqref{eq:as_bxi} holds, there are three equilibrium 
points
\be
\label{eq:equils_chart1}
\begin{array}{lcl}
(y_1,\epsilon_1)&=&(0,0)=:p_1, \\
(y_1,\epsilon_1)&=&\left(\frac{1}{3a_1^*b_1^*},0\right)=:p_2, \\
(y_1,\epsilon_1)&=&\left(\frac{1}{1+a_1^*b_1^*}, 
\frac{1-2 a_1^* b_1^*}{(1+a_1^* b_1^*) (b_1^*-\xi)}\right)=:p_3.\\
\end{array}
\ee
To determine the stability of the eigenvalues we calculate the linearization for
\eqref{eq:bu_kappa1_fol2} with Jacobian(s)
\bea
\label{eq:linearized_r1_syst}
A_{1j}&:=&\left.D_{y_1,\epsilon_1}\left(\begin{array}{c}y_1'\\ \epsilon_1'\\
\end{array}\right)\right|_{p_j}=\\
&=&
\left.\left(\begin{array}{cc}
2-a_1^* b_1^* (12 y_1+\epsilon_1 \kappa)-\epsilon_1 (2 b_1^*+\kappa-2 \xi)&
\kappa-y_1 (\kappa+b_1^* (2+a_1^* \kappa)-2 \xi)\\
-3 a_1^* b_1^* \epsilon_1& 1-3 a_1^* b_1^* y_1-2 b_1^* \epsilon_1+2 \epsilon_1 \xi
\end{array}\nonumber
\right)\right|_{p_j}
\eea
The next result summarizes the relevant stability information for the three 
equilibrium points.

\begin{figure}[htbp]
\psfrag{y1}{$y_1$}
\psfrag{eps1}{$\epsilon_1$}
\psfrag{r1}{$r_1$}
\psfrag{p1}{$p_1$}
\psfrag{p2}{$p_2$}
\psfrag{p3}{$p_3$}
\psfrag{a}{(a)}
\psfrag{b}{(b)}
\psfrag{Ca}{$\bar{C}_0$}
	\centering
		\includegraphics[width=1\textwidth]{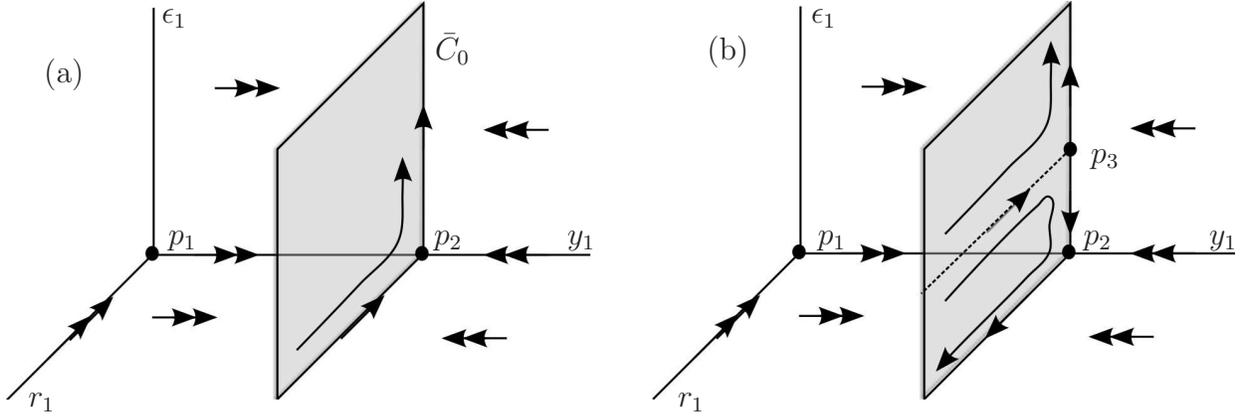}
	\caption{\label{fig:fig4}Sketch of the flows for the chart $\kappa_1$ for 
	\eqref{eq:bu_kappa1_vf_lem}. The variables $a_1=a_1^*$ and $b_1=b_1^*$ are 
	fixed and the case $b_1^*<\xi$ is shown. The gray surface indicates the 
	blow-up of the critical manifold $C_0$ which corresponds to the center manifold 
	$\cM_1$. Double arrows indicate strong attraction or repulsion and single 
	arrows indicate a center flow. (a) $2a_1^*b_1^*<1$: Due to the center flow 
	on $\cM_1$ trajectories approach the sphere and flow upwards near the 
	saddle $p_2$. (b) $2a_1^*b_1^*>1$: The additional equilibrium $p_3$ may prevent 
	the flow up the sphere.}
\end{figure}

\begin{lem}
\label{lem:inv_ey}
Suppose \eqref{eq:as_bxi} holds. The equilibria of \eqref{eq:bu_kappa1_fol2} have 
the following types
\begin{itemize}
 \item $p_1$ is an unstable node with eigenvalues $1$ and $2$,
 \item $p_2$ is center-stable with eigenvalues $-2\kappa$ and $0$. The stable manifold
 associated to the eigenvalue $-2\kappa$ is given by
 \benn
 W^s(p_2)=\{(y_1,\epsilon_1)\in\R^2:\epsilon_1=0,y_1>0\}.
 \eenn
\end{itemize}
Furthermore, $p_3\not\in\bar{\cD}_{\bar{\epsilon}}$ for $(\xi-b_1^*)(2 a_1^* b_1^*-1)<0$, 
$p_3\in\bar{\cD}_{\bar{\epsilon}}$ for $(\xi-b_1^*)(2 a_1^* b_1^*-1)>0$ and $p_3=p_2$ 
when $2 a_1^* b_1^*=1$. If $p_3\in\bar{\cD}_{\bar{\epsilon}}$ then 
\begin{itemize}
 \item $p_3$ is a saddle for $b_1^*<\xi$,
 \item $p_3$ is a sink for $b_1^*>\xi$.
\end{itemize}
\end{lem}

\begin{proof}
The stability results for $p_1$ and $p_2$ follow immediately by looking at the 
$2\times 2$-matrices $A_{11}$ and $A_{12}$ from \eqref{eq:linearized_r1_syst}. The 
stable manifold result for $p_2$ follows from the local information at $p_2$ and
the invariance of the $\{\epsilon_1=0\}$-subspace of \eqref{eq:bu_kappa1_fol2}.
Looking at the sign of the $\epsilon_1$-component of $p_3$ implies when $p_3$ is 
visible in the domain $\bar{\cD}_{\bar{\epsilon}}$. For the stability, we consider 
$A_{13}\in \R^{2\times2}$ from \eqref{eq:linearized_r1_syst}. If $b_1^*<\xi$ then
$p_3$ is a saddle since 
\benn
\det (A_{13})=\frac{(1 - 2 a_1^* b_1^*)^2 \kappa}{(1 + a_1^* b_1^*) (b_1^* -\xi)}<0.
\eenn
Another direct calculation yields, using $(\xi-b_1^*)(2 a_1^* b_1^*-1)>0$, that
\benn
\text{trace}( A_{13})=-4 + \frac{3}{1 + a_1^* b_1^*} + 
\frac{(2 a_1^* b_1^*-1)\kappa}{b_1^* - \xi}<-4 + \frac{3}{1 + a_1^* b_1^*}<0
\eenn
so that for $b_1^*>\xi$ the equilibrium $p_3$ is a sink.
\end{proof}

The equilibrium point $p_3$ passes from the lower-half of the sphere $\epsilon_1<0$ 
to the upper half $\epsilon_1>0$ on the curve $2 a_1^* b_1^*=1$. This implies that 
the flow on the upper half-sphere has two different regimes. Furthermore, the type 
of the equilibria may change based upon the two sub-cases given by $b_1^*<\xi$ 
and $b_1^*>\xi$. This shows the necessity to consider the incoming flow towards the 
fold submanifold very carefully as the variables $(a,b)$ act as additional parameters 
for the invariant foliations in the chart $\kappa_1$. 

Since $p_2$ always exists as an equilibrium point and has one center direction, it 
is necessary to calculate the center manifold. In particular, we return to the system  
\be
\label{eq:bu_kappa1_vf_cm3D}
\begin{array}{lcl}
r_1'&=& r_1\left[-1+\epsilon_1(b^*_1-\xi)+3a^*_1b^*_1y_1\right],\\
y_1'&=& \kappa\epsilon_1(1-y_1(1+a^*_1b^*_1))-
2y_1\left(-1+\epsilon_1(b^*_1-\xi)+3a^*_1b^*_1y_1\right),\\
\epsilon_1'&=&-\epsilon_1\left[-1+\epsilon_1(b^*_1-\xi)+3a^*_1b^*_1y_1\right].\\
\end{array}
\ee 
 
\begin{prop}
\label{prop:twocases}
The center manifold $\cM_1$ for \eqref{eq:bu_kappa1_vf_cm3D} at the 
equilibrium $p_2$ is given as the graph of
\be
\label{eq:M1}
y_1=\frac{1}{3a_1^*b_1^*}+\epsilon_1 \frac{2(\xi-b_1^*)+
\kappa(2 a_1^* b_1^* -1)}{6 a_1^* b_1^*}+
c_{22}\epsilon_1^2+\cO(\epsilon_1^3,\epsilon_1^2r_1,\epsilon_1r_1^2,r_1^3)
\ee
where 
\benn
c_{22}= \frac{\kappa(1+4a_1^*b_1^*)}{24a_1^*b_1^*}(2(b_1^*-\xi)+\kappa(1-2a_1^*b_1^*)). 
\eenn
The flow on $\cM_1$ is
\be
\label{eq:cm_flow1}
\begin{array}{lcl}
r_1'&=& r_1\left[\frac{\kappa(2 a_1^* b_1^* -1)}{2}\epsilon_1+
3a^*_1b^*_1c_{22}\epsilon_1^2+
\cO(\epsilon_1^3,\epsilon_1^2r_1,\epsilon_1r_1^2,r_1^3)\right],\\
\epsilon_1'&=&-\epsilon_1\left[\frac{\kappa(2 a_1^* b_1^* -1)}{2}\epsilon_1
+3a^*_1b^*_1c_{22}\epsilon_1^2+
\cO(\epsilon_1^3,\epsilon_1^2r_1,\epsilon_1r_1^2,r_1^3)\right].\\
\end{array}
\ee
For $b_1^*<\xi$ there are two qualitative cases for the flow \eqref{eq:cm_flow1} 
near the center manifold as shown in Figure \ref{fig:fig4}. For $b_1^*>\xi$ the 
center flow in \ref{fig:fig4}(a) is directed away from the sphere while the 
equilibrium $p_3$ becomes a sink in Figure \ref{fig:fig4}(b).
\end{prop}

\begin{remark}
Note that the center flow is very degenerate when 
$2a_1^*b_1^*=1$ and $b_1^*=\xi$. This corresponds to the case when the initial 
conditions in the chart $\kappa_1$ lie exactly on the degenerate singularity 
$(a,b)=(1/(2\xi),\xi)$. It will be shown in Section \ref{sec:loops} that 
this case will not occur due to the form of the slow flow on $\cC_0$ for the 
periodic orbits we consider in this paper; see Figure \ref{fig:fig8}.  
\end{remark}

\begin{proof}(of Proposition \ref{prop:twocases})
The center manifold calculation is contained in Appendix \ref{ap:cm1} which 
yields \eqref{eq:M1} and consequently also \eqref{eq:cm_flow1}. The results 
in Figure \ref{fig:fig4} follow from Lemma \ref{lem:inv_ry}, Lemma 
\ref{lem:inv_ey} and phase plane analysis of \eqref{eq:cm_flow1} for the 
two cases $2 a_1^* b_1^*>1$ and $2 a_1^* b_1^*<1$. More precisely, 
desingularizing \eqref{eq:cm_flow1} by rescaling time with $1/\epsilon_1$ 
we note that $(r_1,\epsilon_1)=(0,0)=:0$ is saddle for \eqref{eq:cm_flow1}. 
If $2 a_1^* b_1^*<1$ then the stable and unstable eigenspaces are locally 
given by $E^{s}(0)=\{\epsilon_1=0\}$ and $E^{u}(0)=\{r_1=0\}$. The local 
directions are reversed for $2 a_1^* b_1^*>1$ so that $E^{u}(0)=\{\epsilon_1=0\}$ 
and $E^{s}(0)=\{r_1=0\}$.
\end{proof}

Using Proposition \ref{prop:Fenichel_prop_C0}, the correspondence of $\cC_0$ 
and $\cM_1$ via Lemma \ref{lem:inv_ry}, the exponential attraction of 
$\cM_1$ in the $y_1$-direction and the description of the flow 
\eqref{eq:cm_flow1} in Proposition \ref{prop:twocases}, it follows that, 
depending on the invariant foliation determined by the coordinates $(a_1^*,b_1^*)$, 
several cases can occur.

\begin{prop}
\label{prop:approach}
Suppose \eqref{eq:as_bxi} holds, then four cases can occur
\begin{enumerate}
 \item[(C1)] If $2 a_1^* b_1^*<1$, $b_1^*<\xi$ then orbits approach
 $\bar{\cD}\cap\{\bar{r}=0\}$ and flow up the family of spheres into
 the chart $\kappa_2$. 
 \item[(C2)] If $2 a_1^* b_1^*<1$, $b_1^*>\xi$ then orbits approach
 $\bar{\cD}\cap\{\bar{r}=0\}$ and flow up the family of spheres into
 the chart $\kappa_2$ towards the sink $p_3$.
 \item[(C3)] If $2 a_1^* b_1^*>1$ and $b_1^*<\xi$ then orbits may either
 approach $\bar{\cD}\cap\{\bar{r}=0\}$ and flow up the family of spheres 
 into the chart $\kappa_2$, or turn around and continue in the slow flow 
 on $\cC_0$. This case depends upon the initial condition.
 \item[(C4)] If $2 a_1^* b_1^*>1$ and $b_1^*>\xi$ then orbits flow
 away from $\bar{\cD}\cap\{\bar{r}=0\}$.
\end{enumerate}
\end{prop}

For this paper, two of the cases from Proposition \ref{prop:approach} are 
relevant. For the canard and the jump case in Theorem \ref{thm:main_intro}
we are going to need (C1) to track orbits from $\kappa_1$ to $\kappa_2$ when 
they enter a neighbourhood of $\cL_0$. For the canard case, we need (C4) to 
track orbits from $\kappa_2$ to $\kappa_1$ when they leave a neighbourhood of 
$\cL_0$; see also Figure \ref{fig:fig8}. Note that the points $a_1^*,b_1^*$ 
are always the values of the $(a,b)$-coordinates once a vicinity of the center 
manifold $\cM_{1}$ has been reached. Although we shall not need (C2)-(C3) here,
we record them for future work; see also Section \ref{sec:outlook}.\medskip

It remains to investigate the case $b_1^*=\xi$ which will be relevant for the
departure phase for the jump case in Theorem \ref{thm:main_intro}. We shall only
consider the following case
\be
\label{eq:degen_jump_chart1}
b_1^*=\xi, \quad 2a_1^*b_1^*>1, \quad a_1^*-\frac{1}{2\xi}=K>0\qquad \text{for a 
fixed constant $K$ independent of $\epsilon$.}
\ee
This means that we only track orbits transitioning between $\kappa_2$ and 
$\kappa_1$ away from the degenerate singularity $(a,b)=(1/(2\xi),\xi)$ and above 
the curve $\{2ab=1\}$; see Figure \ref{fig:fig8}(b). Under the assumption 
\eqref{eq:degen_jump_chart1} the system \eqref{eq:bu_kappa1_fol2} reduces to
\be
\label{eq:bu_kappa1_fol2a}
\begin{array}{lcl}
y_1'&=& \kappa\epsilon_1(1-y_1(1+a^*_1\xi))-2y_1\left(-1
+3a^*_1\xi y_1\right),\\
\epsilon_1'&=&-\epsilon_1\left[-1+3a^*_1\xi y_1\right].\\
\end{array}
\ee
It is easy to check that \eqref{eq:bu_kappa1_fol2a} only has the two equilibrium
points 
\be
\label{eq:equils_chart1a}
\begin{array}{lcl}
(y_1,\epsilon_1)&=&(0,0)=:p_1, \\
(y_1,\epsilon_1)&=&\left(\frac{1}{3a_1^*b_1^*},0\right)=:p_2, \\
\end{array}
\ee
where it is natural to use the same notation as in \eqref{eq:equils_chart1}. The next 
result is easy  to check by following the same calculations as above using the matrices 
$A_{11}, A_{12}\in\R^{2\times 2}$ for $b_1^*=\xi$.

\begin{lem}
\label{lem:inv_eya}
Suppose \eqref{eq:degen_jump_chart1} holds. The equilibria of \eqref{eq:bu_kappa1_fol2a} are
given by \eqref{eq:equils_chart1a}. $p_1$ is an unstable node with eigenvalues $1$ and $2$.
$p_2$ is center-stable with eigenvalues $-2\kappa$ and $0$. The stable manifold associated 
to the eigenvalue $-2\kappa$ is given by $W^s(p_2)=\{\epsilon_1=0,y_1>0\}$.
Furthermore, the center manifold reduction from Proposition \ref{prop:twocases} is still
valid. 
\end{lem}

\begin{cor}
\label{cor:jump_canard_chart1}
Suppose \eqref{eq:degen_jump_chart1} holds. Then, the case (C4) from Proposition 
\eqref{prop:approach} applies to orbits transitioning from $\kappa_2$ to $\kappa_1$.
\end{cor}

\begin{proof}
By Lemma \ref{lem:inv_eya} we may focus on the dynamics near $p_2$ and consider the 
center manifold reduction \eqref{eq:cm_flow1}. Since $2a_1^*b_1^*>1$ holds by
\eqref{eq:degen_jump_chart1}, the result follows.
\end{proof}

\subsection{Second Chart}
\label{ssec:chart2}

In the last section, we have controlled the flow arriving from $\cC_0$ near $\cL_0$ and
the situation when orbits leave the vicinity of $\cL_0$ towards $\cC_0$. It remains to 
analyze the dynamics in the chart $\kappa_2$ which describes the slow dynamics near 
$\cL_0$. The analysis in this section focuses on 
the system \eqref{eq:Olsen1_bu_res} from Lemma \ref{lem:chart2_flow} which we repeat
here for convenience
\be
\label{eq:Olsen1_bu}
\begin{array}{rcl}
\frac{da_2}{ds}&=& \mu-\alpha a_2 -a_2b_2y_2,\\
\frac{db_2}{ds}&=& \epsilon_b(1-b_2x_2 -a_2b_2y_2),\\
\epsilon^2\frac{dx_2}{ds}&=& 3a_2b_2y_2-x_2^2+(b_2-\xi)x_2+\delta,\\
\epsilon^2\frac{dy_2}{ds}&=& \kappa(x_2^2-y_2-a_2b_2y_2).\\
\end{array}
\ee 
We are not going to make the restriction yet that $\delta=\delta(\epsilon)$
with $\delta(0)=0$ to cover a more general case and view $\delta$ just as a 
parameter. Then the critical manifold of \eqref{eq:Olsen1_bu} is given by
\benn
\cC_{2,0}:=\left\{(a_2,b_2,x_2,y_2)\in\bar{\cD}:
a_2=\frac{x_2^2+x_2(\xi-b_2)-\delta}{b_2(2x_2^2+x_2(b_2-\xi)+\delta)},
y_2=\frac{2x_2^2+x_2(b_2-\xi)+\delta}{3}\right\}.
\eenn
Let $\frac{(1+a_2b_2)(\xi-b_2)}{ 4 a_2 b_2-2}=:l_{2}^\delta(a_2,b_2,\xi)$ and define
\be
\label{eq:def_C2_first}
\begin{array}{lcl}
\cS^{r-}_{2,0}&:=&\cC_{2,0}\cap\{b_2<\xi,x_2>l^\delta_2(a_2,b_2,\xi)\},\\
\cS^{a-}_{2,0}&:=&\cC_{2,0}\cap\{b_2<\xi,x_2<l^\delta_2(a_2,b_2,\xi)\},\\
\cS^{r+}_{2,0}&:=&\cC_{2,0}\cap\{b_2>\xi,x_2<l^\delta_2(a_2,b_2,\xi)\},\\
\cS^{a+}_{2,0}&:=&\cC_{2,0}\cap\{b_2>\xi,x_2>l^\delta_2(a_2,b_2,\xi)\}.\\
\end{array}
\ee
Direct fast-slow systems calculations and Fenichel's Theorem yield the next result;
see also Figure \ref{fig:fig5}.

\begin{prop}
\label{prop:second_chart_prop}
The manifold $\cC_{2,0}$ contains a curve of fold points given by
\benn
\cL^\delta_{2,0}=\left\{(a_2,b_2,x_2,y_2)\in \cC_{2,0}:
x_2=l_{2}^\delta(a_2,b_2,\xi)\right\}.
\eenn
For $a_2\neq1/(2\xi)$ we have that 
\begin{itemize}
 \item $\cS^{r\pm}_{2,0}$ are normally hyperbolic of saddle-type,
 \item $\cS^{a\pm}_{2,0}$ are normally hyperbolic attracting.
\end{itemize}
Furthermore, the manifolds $\cS^{r-}_{2,0}$ and $\cS^{a+}_{2,0}$ are 
unbounded as follows
\begin{itemize}
 \item for $(a_2,b_2,x_2,y_2)\in \cS^{r-}_{2,0}$ we have that 
 $(x_2,y_2)\ra (+\I,+\I)$ when $(2a_2b_2-1)\ra 0$,
 \item for $(a_2,b_2,x_2,y_2)\in \cS^{a+}_{2,0}$ we have that 
 $(x_2,y_2)\ra (+\I,+\I)$ when $(2a_2b_2-1)\ra 0$.
\end{itemize}
For $0<\epsilon\ll 1$ there exist slow manifolds $\cS^{a\pm}_{2,\epsilon}$ 
and $\cS^{r\pm}_{2,\epsilon}$. For $\delta=0$, the curve $\cL^\delta_{2,0}$ 
becomes a line of transcritical points located at $\{b_2=\xi,x_2=0=y_2\}$.  
\end{prop}

\begin{figure}[htbp]
\psfrag{a2}{$a_2$}
\psfrag{b2}{$b_2$}
\psfrag{x2}{$x_2$}
\psfrag{b=xi}{$b_2=\xi$}
\psfrag{gamma}{$\psi_\epsilon$}
\psfrag{2ab=1}{$2a_2b_2=1$}
\psfrag{a}{(a)}
\psfrag{b}{(b)}
\psfrag{c}{(c)}
\psfrag{Sa+}{$\cS^{a+}_{2,0}$}
\psfrag{Sa-}{$\cS^{a-}_{2,0}$}
\psfrag{Sr+}{$\cS^{r+}_{2,0}$}
\psfrag{Sr-}{$\cS^{r-}_{2,0}$}
	\centering
		\includegraphics[width=1\textwidth]{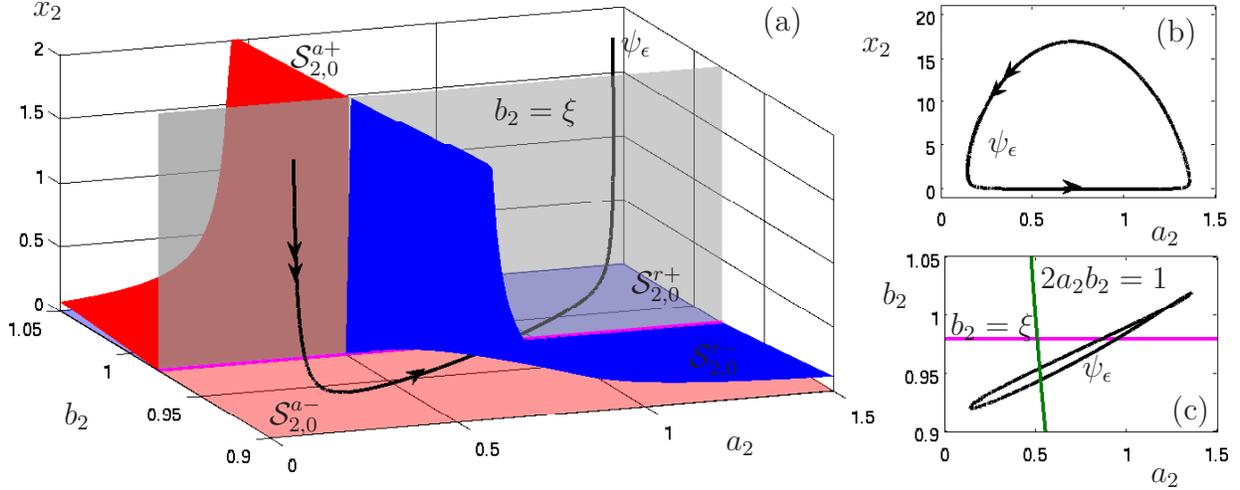}
	\caption{\label{fig:fig5}Illustration for the dynamics and 
	fast-slow decomposition of \eqref{eq:Olsen1_bu}. (a) 
	Three-dimensional projection into $(a_2,b_2,x_2)$-space. 
	For $\delta=0$, the critical manifolds from Proposition 
	\ref{prop:second_chart_prop} (blue=repelling, red=attracting), 
	the transcritical singularities $\cL^\delta_{2,0}$ (magenta) 
	and the hyperplane $\{b_2=\xi\}$ are shown. We also superimpose 
	a truncated periodic solution $\psi_\epsilon$ for $k_1=0.41$ 
	and standard parameter values as shown in Figure \ref{fig:fig1}(c). 
	(b) Projection of the full periodic solution into $(a_2,x_2)$-space. 
	(c) Important curves in the $(a_2,b_2)$-plane.}
\end{figure}

\begin{remark}
We observe that the unbounded structure of the critical manifold for $\delta>0$, $b_2<\xi$ 
resembles closely the autocatalator model \cite{GucwaSzmolyan,GuckenheimerScheper,KuehnUM}. 
Although we shall not need this observation for the types of periodic orbits considered
here, it is likely to be very important for fast dynamics close to the three-dimensional 
submanifold $\{2a_2b_2=1\}$.
\end{remark}

Proposition \ref{prop:second_chart_prop} already shows that we have to expect several cases 
for the dynamics in a neighbourhood of $\{b_2=\xi\}$. For moderate $\delta$ bounded away from zero
and independent of $\epsilon$, we expect the dynamics to be governed by a jump near a fold. 
Although it is relevant to observe this mechanism, we shall not discuss this case here
as it does not occur for the parameter sets considered by Olsen. For sufficiently small 
$\delta$, the transcritical singularity is expected to be relevant and the two 
limiting cases are a jump very close to a transcritical point and canard case with 
maximal delay.\medskip

We start with the singular limit case $\delta=0$. Then, we find that 
\benn
\cS^{a-}_{2,0}=\{(a_2,b_2,x_2,y_2)\in \bar{\cD}:x_2=0=y_0,b_2<\xi\}. 
\eenn
In this case, the slow subsystem on $\cS^{a-}_{2,0}$ is given by 
\be
\label{eq:Olsen1_bu_sf_simple}
\begin{array}{rcl}
\frac{da_2}{ds}&=& \mu-\alpha a_2 ,\\
\frac{db_2}{ds}&=& \epsilon_b.\\
\end{array}
\ee 
The $a_2$-nullcline is $\{a_2=\mu/\alpha\}$. Since we consider $\epsilon_b>0$ as a fixed
parameter we can limit our discussion to the case $a^*<a_2(0)<\mu/\alpha$ here. For
certain types of MMOs we would need $\epsilon_b\ra0$; this case is discussed in Section 
\ref{sec:outlook}. For a given initial condition $(a_2(s_0),b_2(s_0))$ the slow 
subsystem \eqref{eq:Olsen1_bu_sf_simple} can be solved explicitly
\be
\label{eq:plane_sf_solved}
a_2(s)=\frac{\mu}{\alpha}+e^{-\alpha (s-s_0)}
\left(a_2(s_0)-\frac{\mu}{\alpha}\right),\qquad b_2(s)=\epsilon_b (s-s_0)+b_2(s_0).
\ee
It is also interesting to see how $\cC_{2,0}$ asymptotically depends 
upon $\delta$ in the limit $\delta\ra 0$. 

\begin{lem}
\label{lem:delta_bxi}
For $b_2<\xi$ and $\delta\ra 0$ the attracting manifold $\cS_{2,0}^{a-}$ 
is given by
\be
\label{eq:attract1_expand}
\begin{array}{lcl}
x_2&=&\frac{1}{\xi-b_2}\delta+
\frac{1-2a_2b_2}{(1+a_2b_2)(b_2-\xi)^3}\delta^2+\cO(\delta^3),\\
y_2&=&\frac{1}{(1+a_2b_2)(b_2-\xi)^3}\delta^2+\cO(\delta^3),\\
\end{array}
\ee
and the repelling manifold $\cS_{2,0}^{a-}$ is given by
\be
\begin{array}{lcl}
x_2&=&\frac{(1+a_2b_2)(\xi-b_2)}{2a_2b_2-1}+\frac{1}{b_2-\xi}\delta
+\frac{2a_2b_2-1}{(1+a_2b_2)(b_2-\xi)^3}\delta^2+\cO(\delta^3),\\
y_2&=&\frac{(1+a_2b_2)(b_2-\xi)^2}{(1-2a_2b_2)^2}+\frac{2}{1-2a_2b_2}\delta
-\frac{1}{(1+a_2b_2)(b_2-\xi)^2}\delta^2+\cO(\delta^3).\\
\end{array}
\ee
\end{lem}

Normal hyperbolicity of the critical manifold breaks down along the entire
critical manifold $\cC_{2,0}$ once it reaches the hyperplane $b_2=\xi$. Note
that the singularity at $(a_2,b_2)=(1/(2\xi),\xi)$ is again particularly degenerate
and we exclude the set
\be
\label{eq:bad_region}
\cB(\upsilon):=\left\{(a_2,b_2,x_2,y_2)\in\bar{\cD}:
\left(a_2-\frac{1}{2\xi}\right)^2+(b_2-\xi)^2\leq \upsilon\right\}
\ee
for some small $\upsilon>0$ from our analysis as both types of periodic orbits we 
construct do not have a passage at $(a_2,b_2)=(1/(2\xi),\xi)$ in the singular limit 
$\epsilon=0$. As before, we have to make a case
distinction. We assume that orbits only approach a neighborhood of $\{b_2=\xi\}$
via the attracting slow manifolds $\cS^{a-}_{2,0}$; see also Figure \ref{fig:fig5}.
In the next Section \ref{sec:delay} we will consider this case (C1) so that
$2a_1^*b_1^*<1$, $b_1^*<\xi$. For a different case, leading to MMOs, we refer to 
Section \ref{sec:outlook}.

\section{Transcritical Singularities}
\label{sec:delay}

The slow flow \eqref{eq:Olsen1_bu_sf_simple} implies that trajectories reach
the curve of fold points 
\benn
\cL_{2,0}^\delta=\{(a_2,b_2,x_2,y_2)\in\cC_{2,0}:x_2=l^\delta_2(a_2,b_2,\xi)\},
\eenn
which degenerates into a line of transcritical points 
\benn
\cL_{2,0}^0=\{(a_2,b_2,x_2,y_2)\in\cC_{2,0}:x_2=0=y_2,b_2=\xi\}
\eenn
for $\delta=0$. As in Section \ref{ssec:chart2}, we shall view $\delta$ as a 
parameter for now. The goal is to compute the unfolding of the degenerate line 
$\{b_2=\xi\}$ away from the region $\cB(\upsilon)$. In fact, Lemma \ref{lem:delta_bxi} 
already indicates that near $b_2=\xi$, $x_2=0=y_2$ and $\delta=0=\epsilon$ a much 
finer analysis is necessary. In particular, consider the system
\be
\label{eq:Olsen1_bu_scale_translate_main}
\begin{array}{rcl}
\dot{x}_2&=& 3a_2b_2y_2-x_2^2+(b_2-\xi)x_2+\delta,\\
\dot{a}_2&=& \epsilon^2(\mu-\alpha a_2 -a_2b_2y_2),\\
\dot{b}_2&=& \epsilon^2\epsilon_b(1-bx_2 -a_2b_2y_2),\\
\dot{\epsilon}&=& 0,\\
\dot{\delta}&=& 0,\\
\dot{y}_2&=& \kappa(x_2^2-y_2-a_2b_2y_2).\\
\end{array}
\ee 
The $6$-dimensional flow has to be simplified via center manifold reduction near
a line of degenerate equilibrium points 
\benn
\cL:=\{(x_2,a_2,b_2,\epsilon,\delta,y_2)\in\R^6:
x_2=0,a_2=a_0,b_2=\xi,\epsilon=0,\delta=0,y_2=0\}
\eenn
parametrized by $a_0$. The necessary calculations are recorded 
in Appendix \ref{ap:cm2}. 

\begin{prop}
\label{prop:CM2}
For $a_0\neq \frac{1}{2\xi}$, there exists a five-dimensional center
manifold $\cM_2$ for \eqref{eq:Olsen1_bu_scale_translate_main} near
$\cL$. The flow on $\cM_{2}$ is given by
\be
\label{eq:Olsen1_CM2_main}
\begin{array}{rcl}
\epsilon^2\frac{dx_2}{ds}&=& c_2x_2^2+c_1(b_2)x_2+c_0+\cO(3),\\
\frac{da_2}{ds}&=& \mu-\alpha a_0+\cO(2),\\
\frac{db_2}{ds}&=& \epsilon_b+\cO(2),\\
\end{array}
\ee   
with $\frac{d\epsilon}{ds}=0=\frac{d\delta}{ds}$ understood and
$\cO(2)$, $\cO(3)$ denote higher-order terms of order two and
three respectively. Furthermore, the coefficients $c_i$ for
$i=\{0,1,2\}$ are
\benn
\begin{array}{lcl}
c_2&=&\left(\frac{2a_0\xi-1}{1+a_0\xi}\right)>0,\\
c_1(b_2)&=&\left(\frac{-\delta}{\kappa(1+a_0\xi)^2}+b_2-\xi\right),\\
c_0&=& \delta+\frac{\delta^2}{\kappa^2(1+a_0\xi)^3}.\\
\end{array}
\eenn
\end{prop}

For $\delta=0$ the system \eqref{eq:Olsen1_CM2_main} has a line of
transcritical singularities at $b_2=\xi$, as expected. The degenerate
singularity at $a_0=1/(2\xi)$ also appears and causes a sign change
of the $x_2^2$-term, which explains why we restrict to dynamics outside
of $\cB(\upsilon)$. We still have the invariant manifolds
$\cS^{a\pm}_{2,\epsilon}$ and $\cS^{r\pm}_{2,\epsilon}$
(up to higher-order correction terms). Then we define 
\be
\label{eq:delta_scale}
\hat\epsilon:=\epsilon^2\qquad \text{ and }\qquad 
\hat\delta:=\delta/\epsilon^2=\delta(\sqrt{\hat{\epsilon}})/\hat{\epsilon}. 
\ee
Note that we may always assume that $\hat{\delta}$ is bounded, 
{i.e.}~$\hat{\delta}=\cO(1)$ as $\epsilon\ra 0$, due to the assumptions
in Theorem \ref{thm:main_intro} in both of the two main cases as either 
$\delta=\cO(e^{-K_1/\epsilon})$ or $\delta=K_2\epsilon^2$. Using 
\eqref{eq:delta_scale} in \eqref{eq:Olsen1_CM2_main} yields
\be
\label{eq:Olsen1_CM2_main_trunc}
\begin{array}{rcl}
\hat\epsilon\frac{dx_2}{ds}&=&f_2(x_2,a_0,b_2;\hat\epsilon),\\
\frac{da_2}{ds}&=& \mu-\alpha a_0+\cO(2),\\
\frac{db_2}{ds}&=& \epsilon_b+\cO(2).\\
\end{array}
\ee   
where the fast variable vector field is given by
\bea
f_2(x_2,a_0,b_2;\hat\epsilon)&=&\nonumber
\left(\frac{2a_0\xi-1}{1+a_0\xi}\right)x_2^2+
\left(\frac{-\hat\epsilon\hat\delta}{\kappa(1+a_0\xi)^2}+
b_2-\xi\right)x_2\\
&&+\hat\epsilon\hat\delta+
\frac{\hat{\delta}^2\hat{\epsilon}^2}{\kappa^2(1+a_0\xi)^3}+\cO(3).
\label{eq:fast_tc_vf}
\eea
The two normally hyperbolic critical manifolds of \eqref{eq:Olsen1_CM2_main_trunc}
\be
\label{eq:def_slow_C2}
\begin{array}{lcl}
\cS_{2,0}^{r+}&:=&\{f_2(x_2,a_0,b_2;0)=0,x_2=0,b_2>\xi\},\\
\cS_{2,0}^{a-}&:=&\{f_2(x_2,a_0,b_2;0)=0,x_2=0,b_2<\xi\},\\
\end{array}
\ee
are relevant in what follows; note that the definitions agree with
the critical manifolds defined in \eqref{eq:def_C2_first} near the 
transcritical singularity. Next, fix some $a_0$ such that 
\be
\label{eq:tc_away_bad}
a_0>1/(2\xi).
\ee
Consider the system \eqref{eq:Olsen1_CM2_main_trunc}. The differential equation
for $a_2$ does not enter into the local analysis of the unfolding; this can be seen
by applying a change of coordinates 
\benn
a_2=(\mu-\alpha a_0)(\tilde{a}_2+\tilde{b}_2), \qquad b_2=\epsilon_b\tilde{b}_2, 
\eenn
which implies $\tilde{a}_2'=0+\cO(2)$ and and we can focus on 
\be
\label{eq:Olsen1_CM2_main_trunc1}
\begin{array}{rcl}
\hat\epsilon\frac{dx_2}{ds}&=&f_2(x_2,a_0,\epsilon_b \tilde{b}_2;\hat\epsilon)
=:f(x_2,a_0,\tilde{b}_2;\hat{\epsilon}),\\
\frac{d\tilde{b}_2}{ds}&=& 1+\cO(2).\\
\end{array}
\ee 
In principle, we would have to apply another blow-up to \eqref{eq:Olsen1_CM2_main_trunc1}
to unfold the transcritical singularity at $\{x_2=0,\tilde{b}_2=\xi/\epsilon_b\}$. However, 
we are in the fortunate situation that several relevant results are already known for the 
planar transcritical singularity. Here we follow the results from \cite{KruSzm4}. The first
step is to check whether suitable genericity and transversality conditions 
[see (2.2)-(2.3) in \cite{KruSzm4}] hold. The usual transversality condition for the slow variable 
holds trivially since $\frac{d\tilde{b}_2}{ds}=1$. To state the genericity conditions, we use a notational 
simplification and let $p_{\textnormal{tc}}:=(0,a_0,\xi/\epsilon_b;0)$, then the conditions 
[(2.2), \cite{KruSzm4}] are given by
\be
\label{eq:cond_tc_generic}
f(p_{tc})=0,\quad \frac{\partial f}{\partial x_2}(p_{\textnormal{tc}})=0,
\quad \frac{\partial f}{\partial \tilde{b}_2}(p_{\textnormal{tc}})=0,\quad 
\det((D^2f)(p_{\textnormal{tc}}))<0,\quad \frac{\partial^2 f}{\partial x_2^2}(p_{\textnormal{tc}})\neq0.
\quad 
\ee
where $D^2f\in\R^{2\times 2}$ denotes Hessian matrix with respect to the variables $(x_2,\tilde{b}_2)$.
It is easy to check that the conditions \eqref{eq:cond_tc_generic} hold.
Near a transcritical point, three main dynamical regimes may occur. To distinguish these cases, one
possibility is to calculate a particular constant $\lambda_{\textnormal{tc}}$. To define it, it is 
helpful to consider the following auxiliary constants
\benn
c_{xx}:=\frac12 \frac{\partial^2 f}{\partial x_2^2}(p_{\textnormal{tc}}),\quad 
c_{xb}:=\frac12 \frac{\partial^2 f}{\partial x_2 \partial \tilde{b}_2}(p_{\textnormal{tc}}),\quad 
c_{bb}:=\frac12 \frac{\partial^2 f}{\partial \tilde{b}_2^2}(p_{\textnormal{tc}}),\quad 
c_{\epsilon}:=\frac{\partial f}{\partial \hat{\epsilon}}(p_{\textnormal{tc}}). 
\eenn
Then the key constant $\lambda_{\textnormal{tc}}$ [Lemma 2.1, \cite{KruSzm4}], in the context
of \eqref{eq:Olsen1_CM2_main_trunc1}, is given by
\beann
\lambda_{\textnormal{tc}}&=&\frac{1}{\sqrt{c_{xb}^2-c_{bb} c_{xx}}}
\left(c_\epsilon c_{xx}+c_{xb}\right)\\
&=& \frac{1}{\sqrt{\epsilon_b^2-0\cdot c_{xx}}}
\left(\frac{\hat{\delta}}{2}\cdot\frac{2a_0\xi-1}{1+a_0\xi}+\epsilon_b\right).
\eeann
The next result explains the case distinction from Theorem \ref{thm:main_intro}.

\begin{prop}
\label{prop:tc_canards}
Consider \eqref{eq:Olsen1_CM2_main_trunc} and suppose \eqref{eq:tc_away_bad} holds. 
Let $\gamma_{\hat\epsilon}=\gamma_{\hat\epsilon}(s)$ denote a trajectory of 
\eqref{eq:Olsen1_CM2_main_trunc} starting at some $s_0$ exponentially close 
to $\cS^{a-}_{2,\epsilon}$.
Then there exists $\hat{\epsilon}_0>0$ such that for all 
$\hat{\epsilon}\in(0,\hat{\epsilon}_0]$ the following cases may occur
\begin{enumerate}
 \item \textbf{Canard case}: Suppose $\hat{\delta}=\cO(e^{-K_1/\hat{\epsilon}})$ 
 and $K_1>0$ is some fixed constant independent of $\hat{\epsilon}$. Then 
 $\gamma_{\hat{\epsilon}}$ has a canard segment {i.e.}~it is $\cO(\hat{\epsilon})$-close 
 to $\cS^{r+}_{2,0}$ for a time $s^*=\cO(1)$, $s^*>0$ as $\epsilon\ra 0$.
 \item \textbf{(Transcritical) jump case:} Suppose $\hat\delta=K_2$, 
 $K_2>0$ and $K_2$ is fixed as $\hat{\epsilon}\ra 0$. Then 
 $\gamma_{\hat{\epsilon}}$ does not have a canard segment and 
 leaves $\{x_2=0,y_2=0\}$ in an $\hat{\epsilon}$-dependent neighbourhood 
 $\tilde{\cN}(\hat{\epsilon})$ of $\{\tilde{b}_2=\xi/\epsilon_b\}$ such that 
 $d_{\textnormal{H}}(\tilde{\cN}(\epsilon),\{b_2=\xi/\epsilon_b\})\ra 0$ 
 as $\epsilon\ra 0$. 
\end{enumerate}
\end{prop}

\begin{proof}
Basically, the result follows from [Theorem 2.1, \cite{KruSzm4}] and the observation
that for $\hat{\delta}=\cO(e^{-K_1/\hat{\epsilon}})$ the two manifolds 
$\cS_{2,\epsilon}^{a-}$ and $\cS_{2,\epsilon}^{r+}$ are exponentially close. More
precisely, we start with the jump case and observe that
\benn
\lambda_{\textnormal{tc}}=1+\frac{K_2}{2\epsilon_b}\cdot\frac{2a_0\xi-1}{1+a_0\xi}>1
\eenn
since $2a_0\xi-1>0$, $K_2>0$, $1+a_0\xi>0$ and $\epsilon_b>0$. This implies that
we are in the situation of [Theorem 2.1(a), \cite{KruSzm4}], which implies that a 
jump occurs near the transcritical point as described in the result we want to prove.
The exchange-of-stability case [Theorem 2.1(b), \cite{KruSzm4}] does not occur as
we always have $\lambda_{\textnormal{tc}}\geq 1$. For the canard case, observe that 
for $\hat{\delta}=0$ the submanifold $\{x_2=0=y_2\}$ is invariant for 
\eqref{eq:Olsen1_bu_scale_translate_main}, and hence $\{x_2=0\}$ is invariant for
\eqref{eq:Olsen1_CM2_main_trunc}. Therefore, when 
$\hat{\delta}=\cO(e^{-K_1/\hat{\epsilon}})$ holds, it follows that 
$\cS_{2,\epsilon}^{a-}$ and $\cS_{2,\epsilon}^{r+}$ are exponentially close. This 
yields the result for the canard case.
\end{proof}

For the jump case, there is no further analysis required. Proposition 
\ref{prop:tc_canards} implies for this case that we can use the coordinates
\be
\label{eq:ic_cond_jump}
(a_2(s_1),b_2(s_1),y_2(s_1))=
\left(\frac\mu\alpha+e^{-\frac{\alpha}{\epsilon_b}(\xi-b_2(s_0))}
\left(a_2(s_0)-\frac\mu\alpha\right),\xi,0\right)
\ee
as initial conditions, up to an $\epsilon$-dependent error term, for the 
slow flow on the critical manifold $\cC_0$ on which large loops occur. 
This regime is considered in Section \ref{sec:loops}.\medskip

For the canard case, Proposition \ref{prop:tc_canards} implies that
trajectories exponentially close to $\cS_{2,\epsilon}^{a-}$ will locally 
experience maximal delay \cite{Neishtadt1}. However, to verify that there
is also global maximal delay, we have to investigate the fast subsystem 
of \eqref{eq:Olsen1_bu_scale_translate_main} linearized around 
$\{x_2=0=y_2\}$. This linearized system is given by
\be
\label{eq:Olsen1_bu_scale_fss_linear}
\frac{d}{ds}\left(\begin{array}{c}
X_2\\
Y_2\\
\end{array}\right)
=\underbrace{\left(\begin{array}{cc}
(b_2-\xi) & 3a_2b_2,\\
0& -(1+a_2b_2),\\
\end{array}\right)}_{=:A_{\textnormal{fs}}}
\left(\begin{array}{c} X_2\\ Y_2\\
\end{array}\right).
\ee 
The next result follows from a direct calculation.

\begin{lem}
\label{lem:fss_lin_canard}
The matrix $A_{\textnormal{fs}}$ has eigenvalues 
$\lambda_{\textnormal{fs},1}=b_2-\xi$, 
$\lambda_{\textnormal{fs},2}=-(1+a_2b_2)$ with associated eigenvectors
\benn
\Lambda_{\textnormal{fs},1}=\left(\begin{array}{c}
1\\
0\\
\end{array}\right)\qquad \text{and}\qquad 
\Lambda_{\textnormal{fs},2}=\left(\begin{array}{c}
-\frac{3a_2b_2}{1+b_2+a_2b_2-\xi}\\
1\\
\end{array}\right).
\eenn
Furthermore, $\lambda_{\textnormal{fs},1}$ is the critical eigenvalue 
under suitable conditions {i.e.}
\begin{itemize}
\item[(E1)] $0\geq\lambda_{\textnormal{fs},1}>\lambda_{\textnormal{fs},2}$ 
as long as $1+a_2b_2+b_2>\xi$, $b_2\leq \xi$,
\item[(E2)] $0<\lambda_{\textnormal{fs},1}<|\lambda_{\textnormal{fs},2}|$
as long as $b_2> \xi$, $-1-a_2b_2+b_2<\xi$.
\end{itemize}
\end{lem}

For standard Olsen parameter values we have $\xi=0.98$. All candidate 
orbits we are going to construct below are going to satisfy the 
conditions on $a_2,b_2$ (respectively $a,b$) globally. Therefore,  
the eigendirection associated to $\lambda_{\textnormal{fs},1}$ is 
critical, in the sense that the eigenvalue $\lambda_{\textnormal{fs},1}$
is also the weak eigenvalue compared to $\lambda_{\textnormal{fs},2}$.\medskip

We may now return to the calculation of the delay time for the canard case.
The delay time depends upon the initial conditions using the way-in way-out 
function \cite{Neishtadt2}  
\benn
\Pi(\gamma_{0}(s)):=\int_{s_0}^s (D_{x_2}f_2)(\gamma_{0}(\eta);0)~d\eta.
\eenn
where $f_2$ is the fast vector field variable from \eqref{eq:Olsen1_CM2_main_trunc}. 
Note that $\Pi(\gamma_0(s_0))=0$. 

\begin{prop}
\label{prop:delay1}
Suppose \eqref{eq:tc_away_bad} holds. Assume $\hat\delta(\hat\epsilon)=\cO(e^{-K_1/\hat\epsilon})$ and 
$\gamma_{\hat\epsilon}(s_0)$ is in an $\cO(1)$-neighborhood and the 
fast-flow basin of attraction for $\cS_{2,\epsilon}^{a-}$. Suppose 
$\gamma_{0}(s^*)\in\cL$ for some $s^*>s_0$ then for times $s$ with 
\benn
s_0+\cO(\hat\epsilon\ln\hat\epsilon)\leq s \leq \Pi(s_1)+\cO(\hat\epsilon\ln\hat\epsilon)
\eenn
the trajectory $\gamma_{\hat\epsilon}(s)$ is in an $\cO(\hat\epsilon)$-neighborhood 
of $\cS_{2,0}^{r+}$ where $\Pi(s_1)=0$ and $s_0<s^*<s_1$.
\end{prop}

\begin{proof}
Since solutions do not leave the positive quadrant and we start in $\bar{\cD}$, a trajectory 
can never enter the stable invariant submanifold which is tangent to the local stable eigenspace
associated to $\lambda_{\textnormal{fs},2}$ since $\Lambda_{\textnormal{fs},2}$ has always 
one negative component. Furthermore, the eigenvalue $\lambda_{\textnormal{fs},1}$ is always
critical so that we may apply a previous result [Theorem 2.4, \cite{Schecter5}; 
see also \cite{Liu5} where the result was proven first]. The logarithmic corrections of the 
transition time, calculated using the way-in way-out function, is a direct consequence of 
the calculation in \cite{Neishtadt2}.
\end{proof}

In analogy to the jump case \eqref{eq:ic_cond_jump}, we also want to compute the 
departure point $\gamma_0(s_1)$ in $\cS_{2,0}^{r+}$ for the canard case in the 
singular limit. In view of Proposition \ref{prop:delay1}, we have to evaluate 
the way-in way-out function.  

\begin{cor}
\label{cor:delay1}
Under the same assumptions as in Proposition \ref{prop:delay1} the point 
$\gamma_0(s_1)$ is given by  
\be
\label{eq:jump_after_delay}
\gamma_0(s_1)=\left(\frac\mu\alpha+e^{-\alpha(2/\epsilon_b(\xi-b_2(s_0)))}\left(a_2(s_0)
-\frac\mu\alpha\right),2\xi-b_2(s_0),0,0\right).
\ee
\end{cor}

\begin{proof}
A direct calculation yields
\beann
\int_{s_0}^s D_{x_2}f_2(\gamma_0(\eta);0)~d\eta&=&\int_{s_0}^s (b_2(\eta)-\xi) d\eta\\
&=&\int_{s_0}^s (\epsilon_b\eta+b_2(s_0)-\xi) d\eta\\
&=&\epsilon_b\frac12(s-s_0)^2+(b_2(s_0)-\xi)(s-s_0).
\eeann
For $s_1>s_0$ we find that $\Pi(\gamma_0(s_1))=0$ if 
$s_1=2/\epsilon_b(\xi-b_2(s_0))+s_0$. Using the solution 
\eqref{eq:plane_sf_solved} of the slow subsystem 
\eqref{eq:Olsen1_bu_sf_simple} and substituting the result 
for $s_1$ yields \eqref{eq:jump_after_delay}.
\end{proof}

Proposition \ref{prop:tc_canards} and Corollary \ref{cor:delay1} imply that we can use 
the coordinates
\be
\label{eq:ic_cond_canard}
(a_2(s_1),b_2(s_1),y_2(s_1))=
\left(\frac\mu\alpha+e^{-\alpha(2/\epsilon_b(\xi-b_2(s_0)))}
\left(a_2(s_0)-\frac\mu\alpha\right),2\xi-b_2(s_0),0\right)
\ee
as initial condition, up to an $\epsilon$-dependent error term, in the canard case 
for the slow flow on the critical manifold $\cC_0$ on which large loops occur. 
This regime is considered in Section \ref{sec:loops}.

\section{Large Loops}
\label{sec:loops}

We return to the analysis of the slow flow on the critical manifold $\cC_0$
from Section \eqref{sec:tr_res}. The flow is given by
\be
\label{eq:Olsen2_sf_big}
\begin{array}{rcl}
\frac{da}{d\tau}&=& -aby,\\
\frac{db}{d\tau}&=& -\epsilon_baby,\\
\frac{dy}{d\tau}&=&\kappa(2ab-1) y.\\
\end{array}
\ee

\begin{prop}
\label{prop:large}
The slow flow \eqref{eq:Olsen2_sf_big} is solved by
\be
\label{eq:solve_large}
b=\epsilon_b a+K_1\qquad \text{and} \qquad 
y=K_2+\kappa \left(-2 a+\frac{\ln a}{K_1}-\frac{\ln(K_1+a \epsilon_b)}{K_1}\right)
\ee
for constants $K_{1,2}$ to be determined from the initial conditions.
\end{prop}

\begin{proof}
From the first two equations in \eqref{eq:Olsen2_sf_big} it follows that 
\benn
\frac{db}{da}=\epsilon_b \qquad \Rightarrow \quad b=\epsilon_b a+K_1
\eenn
for some constant $K_1$. Inserting this result in the first and 
third equation of \eqref{eq:Olsen2_sf_big} yields
\be
\label{eq:Ysolve}
\frac{dy}{da}=-\kappa \frac{2a(\epsilon_ba+K_1)-1}{a(\epsilon_b a+K_1)}
=-2\kappa +\frac{\kappa}{a(\epsilon_b a+K_1)}.
\ee
The result follows upon solving \eqref{eq:Ysolve} explicitly.
\end{proof}

Proposition \ref{prop:large} resolves the global large return dynamics
for $x,y$. However, we still need to consider the solution \eqref{eq:solve_large}
in even more detail. It is going to be helpful to extract as much relevant
information from the explicit solution \eqref{eq:solve_large} analytically
as possible to prove results about oscillations in the Olsen model. The main
case we are interested in is an initial condition 
\be
\label{eq:starting_point_large}
(a(0),b(0),y(0))=(\alpha_1,\beta_1,0)
\ee
corresponding to a singular loop starting at a fold point. The 
choice of subscripts in \eqref{eq:starting_point_large} will become clear
in Section \ref{sec:candidate}. We are most
interested in the two cases when \eqref{eq:starting_point_large} is either
given by \eqref{eq:ic_cond_jump} for the jump case, or by \eqref{eq:ic_cond_canard}
in the canard case.\medskip

From \eqref{eq:starting_point_large} and \eqref{eq:solve_large} it follows
that $\beta_1-\epsilon_b \alpha_1=K_1$. Furthermore, for \eqref{eq:Ysolve} the
initial condition is $y(\alpha_1)=0$. Direct calculations yield
\be
\label{eq:sol_loop}
y(a)=\frac{\kappa}{\beta_1-\epsilon_b \alpha_1}\left[2(a-\alpha_1)
(\alpha_1\epsilon_b-\beta_1)+
\ln\left(\frac{\beta_1a}{\alpha_1(\beta_1+\epsilon_b(a-\alpha_1))}\right)\right],
\ee
where we have to assume $\alpha_1(\beta_1+\epsilon_b(a-\alpha_1))\neq 0$; this 
last assumption will always satisfied for standard Olsen parameter values in the region
of interest for candidate orbits we want to construct. Indeed, note that we always 
have positive $a,\alpha_1=\cO(1)$ and $b>b^*$ is bounded away from zero by a suitable 
constant $b^*>0$.\medskip 

Now observe carefully that the singular (or candidate) loops are restricted to
a family of invariant lines upon projection into the $(a,b)$-plane
\be
\label{eq:inv_lines}
\{(a,b,x,y)\in \cD:x=0=y,b=\epsilon_ba+\beta_1-\epsilon_b\alpha_1\}.
\ee
We collect some important information on the function $y(a)$.

\begin{lem}
\label{lem:returns}
Considering \eqref{eq:sol_loop} we have
\benn
y'(a)=\frac{(1 - 2 a (\beta_1 + (a - \alpha_1) \epsilon_b) \kappa}
{a (\beta_1 + (a - \alpha_1) \epsilon_b)},\qquad \text{so that 
$y'(\alpha_1)=\frac{\kappa(1-2\alpha_1\beta_1)}{\alpha_1\beta_1}$.} 
\eenn
Assume $\beta_1-\alpha_1\epsilon_b>0$ and standard parameter values 
then $y(a)$ has local extrema at
\benn
a_{\pm}=\frac{2 \alpha_1 \epsilon_b -2 \beta_1 \pm 
\sqrt{8 \epsilon_b + (2 \alpha_1 \epsilon_b-2 \beta_1 )^2}}{4 \epsilon_b}
\eenn
with $a_+>0$, $a_-<0$, $y(a_+)\geq0$. For $a\in[0,+\I)$ one finds
that $a_{+}$ is a global maximum, $a_+<\alpha_1$ for $2\alpha_1\beta_1-1>0$,
$y(a_+)=0$ if and only if $2\alpha_1\beta_1-1=0$ and $y(a_+)>0$ for
$2\alpha_1\beta_1-1\neq0$. Furthermore, we have the asymptotics
\benn
\lim_{a\ra 0^+} y(a)=-\I \qquad \text{and} \qquad \lim_{a\ra+\I}y(a)=-\I.
\eenn
\end{lem}
 
We restrict to loops for $y(\alpha_1)=0$ when $2\alpha_1\beta_1-1>0$ based 
upon the results in Section \ref{sec:delay}. From Lemma \ref{lem:returns}
it follows that there exists another zero $\alpha_2$ such that $y(\alpha_2)=0$
and $\alpha_2<a_+<\alpha_1$. Therefore, Lemma \ref{lem:returns} provides a 
rigorous justification for the trajectories shown in Figure \ref{fig:fig7}(a) which
make large excursions, with a single maximum, in $\cC_0$. To compute the
landing point $\alpha_2$ we must solve the equation  
\be
\label{eq:main_sol_problem}
y(\alpha_2)=0\quad \Leftrightarrow \quad 2(\alpha_2-\alpha_1)
(\beta_1-\alpha_1\epsilon_b)=
\ln\left(\frac{\beta_1\alpha_2}{\alpha_1(\beta_1+\epsilon_b(\alpha_2-\alpha_1))}\right)
\ee
which is transcendental and the solutions cannot be given in closed form. Despite
this problem one can still use \eqref{eq:main_sol_problem} to construct candidate orbits.

\section{Construction of Candidate Orbits}
\label{sec:candidate}

In this section we construct global candidate orbits for the jump case and the canard 
case. Since there are two different starting points for the large loops to consider, 
{i.e.}~either \eqref{eq:ic_cond_jump} or \eqref{eq:ic_cond_canard}, we subdivide the 
following discussion into two cases. 

\subsection{A Canard Candidate}
\label{ssec:canard_candidate}

The canard case is more difficult so we shall discuss it first, and in more detail. 
The candidate orbit we aim to construct consists of a concatenation of a slow flow
segment defined by \eqref{eq:Olsen1_bu_sf_simple} on the time scale $s$ with maximal
delay and a fast segment on the time scale $\tau=\epsilon^{-2}s$ for \eqref{eq:Olsen2}. 
Both segments are constructed in the singular limit for $\epsilon=0$ with $\epsilon_b>0$.
The fast segment is itself a slow segment for the subsystem \eqref{eq:Olsen2_sf_big} on
the attracting critical manifold $\cC_0$; see also Figure \ref{fig:fig7}(a).\medskip

Let $(\alpha_0,\beta_0)$ denote an initial condition for \eqref{eq:Olsen1_bu_sf_simple}
with $2\alpha_0\beta_0<1$ and $\beta_0<\xi$. By Corollary \ref{cor:delay1} the
maximal delay point $(\alpha_1,\beta_1)$ is given by 
\bea
\alpha_1&=&\frac\mu\alpha+e^{-\alpha(2/\epsilon_b(\xi-\beta_0))}
\left(\alpha_0-\frac\mu\alpha\right),\label{eq:cand1}\\
\beta_1&=&2\xi-\beta_0.\label{eq:cand2}
\eea
Augmenting this point by the trivial condition $y=0$ gives $(\alpha_1,\beta_1,0)=(a,b,y)$
which is the initial condition for the slow flow \eqref{eq:Olsen2_sf_big} governing
the large loop. By Proposition \ref{prop:large} and equation \eqref{eq:main_sol_problem}
the conditions 
\bea
0&=&2(\alpha_2-\alpha_1)(\beta_1-\alpha_1\epsilon_b)
-\ln\left(\frac{\beta_1\alpha_2}{\alpha_1(\beta_1+\epsilon_b(\alpha_2-\alpha_1))}\right),
\label{eq:cand3}\\
\beta_1&=&\epsilon_b \alpha_1+\beta_2-\epsilon_b\alpha_2\label{eq:cand4}
\eea
follow, where \eqref{eq:cand4} is the requirement to lie in a single invariant line
\eqref{eq:inv_lines}. For a periodic candidate orbit we must have 
\be
\label{eq:cand_periodic}
(\alpha_0,\beta_0)\stackrel{!}{=}(\alpha_2,\beta_2).
\ee
Substituting \eqref{eq:cand_periodic} into \eqref{eq:cand1}-\eqref{eq:cand4} yields
a nonlinear system of four algebraic equations in four unknowns
$(\alpha_0,\beta_0,\alpha_1,\beta_1)$. 

\begin{lem}
\label{lem:key_algebra1}
The system \eqref{eq:cand1}-\eqref{eq:cand_periodic} can be simplified to a single
algebraic equation for $\beta_0$ given by
\bea
\label{eq:b0_equation}
0&=&4(\beta_0-\xi)(\epsilon_b\mu-\alpha\xi)+4(\beta_0-\xi)w_c(\beta_0)+\\
&&\alpha\epsilon_b \ln\left[\frac{(2\xi-\beta_0)(\beta_0\alpha+\epsilon_b\mu
-\alpha\xi+w_c(\beta_0))}{\beta_0 \left(\epsilon_b\mu-\alpha \beta_0+\alpha\xi
+w_c(\beta_0)\right)}\right]=:W_c(\beta_0)\nonumber
\eea
where $w_c(\beta_0):=\alpha(\beta_0-\xi)
\coth\left[\frac{\alpha(\xi-\beta_0)}{\epsilon_b}\right]$.
\end{lem}

\begin{proof}
Replace $\beta_1$ in \eqref{eq:cand3} and \eqref{eq:cand4} using \eqref{eq:cand2}
which only depends on $\beta_0$. Then replace $\alpha_1$ in \eqref{eq:cand3}
and \eqref{eq:cand1} using \eqref{eq:cand4}
{i.e.}~$(2\xi-2\beta_0+\epsilon_b\alpha_0)/\epsilon_b= \alpha_1$. Then notice
that \eqref{eq:cand1} can be solved for $\alpha_0$ 
\be
\label{eq:magic_formula}
\alpha_0 = \frac{\beta_0 \alpha + \epsilon_b \mu
- \alpha \xi + \alpha (\beta_0 - \xi)
\coth[\alpha (\xi-\beta_0)/\epsilon_b]}{\alpha \epsilon_b}.
\ee
Substituting \eqref{eq:magic_formula} into \eqref{eq:cand3} gives the result.   
\end{proof}

Hence we have to determine whether $W_c(\beta_0)$ has a zero, which also satisfies
the relevant constraints as an arrival point for a large loop, {i.e.}~we need 
\be
\label{eq:constraints}
\beta_0<\xi \qquad \text{and}\qquad
2\alpha_0\beta_0=\frac{2\beta_0}{\alpha\epsilon_b}\left(\beta_0\alpha
+\epsilon_b \mu -\alpha\xi +w_c(\beta_0)\right)<1.
\ee
This requires a better understanding of the function $W_c$. Having reduced the 
problem to a single algebraic equation, we could investigate $W_c$ numerically.
However, it is even possible to obtain analytical results. We view $\mu$ as 
a parameter that we may adjust to find the required root.

\begin{lem}
\label{lem:find_canard_jump}
The following properties hold:
\begin{itemize}
 \item[(P1)] $W_c(\xi)=0$.
 \item[(P2)] $\frac{dW_c}{d\beta_0}(\xi)=W_c'(\xi)=
 \frac2\xi (\epsilon_b-\epsilon_b \mu+\alpha \xi) (\alpha-2 (\mu-1) \xi)(\mu-1)^{-1}$.
 \item[(P3)] Suppose that $\epsilon_b,\alpha,\xi,\mu>0$. Then $W'_c(\xi)<0$ if and only if
 one of the following three cases holds
 \benn
 \begin{array}{cl}
 (i)\qquad &0<\mu<1,\quad \alpha>-2\xi+2\mu\xi,\\
 (ii)\qquad &1<\mu<\frac{\alpha+2\xi}{2\xi},\quad \epsilon_b>\frac{\alpha\xi}{\mu-1},\\
 (iii)\qquad &\mu>\frac{\alpha+2\xi}{2\xi},\quad0<\epsilon_b<\frac{\alpha\xi}{\mu-1}.\\
 \end{array}
 \eenn
 \end{itemize}
Consider $\mu$ as a parameter and otherwise standard parameter values from Table 
\ref{tab:tab2} ($k_1=0.41$). Then the following results hold:
\begin{itemize}
 \item[(P4)] If $0< \mu<1$ then $W_c(\beta_0)\neq0$ for $\beta_0\in(0,\xi)$.
 \item[(P5)] There exists an open set $(\mu_1,\mu_2)$ with $1<\mu_1<\mu_2$ such that 
 if $\mu\in(\mu_1,\mu_2)$ then $W_c(\beta_0^*)=0$ for some $0<\beta_0^*<\xi$.  
\end{itemize}

\end{lem}

\begin{proof}
(P1) and (P2) follow from a direct limit calculation. (P3) is a corollary
of (P2). (P4) can be proven via a geometric argument: The invariant lines
\eqref{eq:inv_lines} have slope $m_1=\epsilon_b$ as functions of $a$. The slow
flow \eqref{eq:Olsen1_bu_sf_simple} is affine with direction
$(\mu-\alpha a,\epsilon_b)^T$ so that locally we have a slope
$m_2=\epsilon_b/(\mu-\alpha a)$. A canard candidate of the prescribed form
certainly requires $m_1>m_2$ at the point $(\alpha_0,\beta_0)$ as a
trajectory of \eqref{eq:Olsen1_bu_sf_simple} must intersect a single
invariant line \eqref{eq:inv_lines} twice. Hence, 
$\epsilon_b>\epsilon_b/(\mu-\alpha_0 a)$ and so $\mu-\alpha_0 a>1$;
even if $\alpha_0$ is very small we need at least $\mu>1$ which proves (P3).

Regarding (P5), we first observe that (P4) implies we have to restrict 
to the case $\mu\geq1$ to find a root. Then we consider (P3) and observe 
that $W'_c(\xi)<0$ if and only if (P3)(iii) holds (which actually yields 
a bound $\mu>233/196$). If we can show that there exists $\beta_0\in(0,\xi)$
such that $W_c(\beta_0)<0$, then the intermediate value theorem will yield
the required root, as well as the open set of $\mu$-parameter values. To show
this, consider one part of the argument of the logarithmic term in $W_c$ given by
\benn
w_{\textnormal{aux}}(\beta_0):=\epsilon_b\mu+\alpha(\beta_0-\xi)+w_c(\beta_0).
\eenn
We may directly check that $w_{\textnormal{aux}}(\xi)=\epsilon_b(\mu-1)>0$ and there
exists $\beta_0<\xi$ such that $w_{\textnormal{aux}}(\beta_0)<0$ 
{e.g.}~$w_{\textnormal{aux}}(\beta_0-\epsilon_b/\alpha)=\mu-1-\coth(1)$ so that
there exists an open set of $\mu$-values for which $\mu>233/196$ and $\mu-1-\coth(1)<0$. 
The intermediate value theorem implies that there exists $\beta_{00}$ such that 
$w_{\textnormal{aux}}(\beta_{00})=0$. Using another direct calculation, we see that
the term $w_{\textnormal{aux}}(\beta_0)$ dominates in the exponential and it
follows that $\lim_{\beta_0\ra \beta_{00}}W_c(\beta_0)=-\I$. Hence, there exists 
a $\beta_0^*$ with $\beta_0^*<\xi$ such that $W(\beta_0^*)=0$.
\end{proof}

Of course, the result (P5) above is not very explicit and could potentially be 
improved. However, we do not think it is possible to provide a full classification
of all periodic solutions based upon all the system parameters analytically. To 
explore various quantitative bounds for parameter ranges, it seems more adequate 
to use numerical methods, such as numerical continuation 
\cite{DesrochesKrauskopfOsinga1,Desrochesetal}. Here 
we only provide a proof of the main geometric structure. 
Lemma \ref{lem:find_canard_jump} implies the existence of a candidate
orbit for the canard case of Theorem \ref{thm:main_intro}.

\begin{cor}
\label{cor:sing_orbit}
There exists an open set $(\mu_{c,1},\mu_{c,2})$ with $\mu_{c,1}<\mu_{c,2}$ such that the 
Olsen model \eqref{eq:Olsen2}, for $\mu\in(\mu_{c,1},\mu_{c,2})$ and otherwise standard 
parameter values from Table \ref{tab:tab1} ($k_1=0.41$), has a periodic 
candidate orbit $\psi_0$. It consists of two segments, one for the slow flow 
\eqref{eq:Olsen1_bu_sf_simple} including a canard segment and one consisting
of a large loop defined by \eqref{eq:Olsen2_sf_big}.
\end{cor}

Corollary \ref{cor:sing_orbit} is a singular limit result for $\epsilon=0$. 
Therefore, we still have to show that the candidate orbit indeed perturbs to
an actual periodic orbit for $0<\epsilon\ll1$. This step, which is actually 
the reason why we use the blow-up technique, is carried out in Section 
\ref{sec:retmap}. We refer to the candidate orbit from Corollary \ref{cor:sing_orbit}
as a candidate of a non-classical relaxation oscillation. Indeed, comparing
with the classical relaxation oscillation case \cite{Grasman}, one immediately
notices that our construction here still has a `fast' phase corresponding to the
large loop and a `slow' phase corresponding to a sliding-type motion near the
fold locus. However, the critical manifold structure(s) as well as the fast-slow
decomposition differ substantially from the cubic or S-shaped critical manifold of
classical relaxation oscillations; see also Figure \ref{fig:fig5}. 

\subsection{A Jump Candidate}
\label{ssec:jump_candidate}

The next step is to also consider the jump case from Proposition \ref{prop:tc_canards}
in combination with the large loops. The jump candidate orbit consists of a concatenation 
of a slow flow segment defined by \eqref{eq:Olsen1_bu_sf_simple} on the time scale 
$s$ up to $b_2=\xi$ and a fast segment on the time scale $\tau=\epsilon^{-2}s$ for
\eqref{eq:Olsen2}; see also Figure \ref{fig:fig7}(a).\medskip

As in Section \ref{ssec:canard_candidate}, let $(\alpha_0,\beta_0)$ denote an
initial condition for \eqref{eq:Olsen1_bu_sf_simple} with $2\alpha_0\beta_0<1$
and $\beta_0<\xi$. The departure point $(\alpha_1,\beta_1)$ for the jump case 
is calculated in \eqref{eq:ic_cond_jump}. Analogously, to 
\eqref{eq:cand1}-\eqref{eq:cand4} we get four algebraic equations
\bea
\alpha_1&=&\frac\mu\alpha+e^{-\frac{\alpha}{\epsilon_b}(\xi-\beta_0)}
\left(\alpha_0-\frac\mu\alpha\right),\label{eq:cand1a}\\
\beta_1&=&\xi,\label{eq:cand2a}\\
0&=&2(\alpha_2-\alpha_1)(\beta_1-\alpha_1\epsilon_b)
-\ln\left(\frac{\beta_1\alpha_2}{\alpha_1(\beta_1+\epsilon_b(\alpha_2-\alpha_1))}\right),
\label{eq:cand3a}\\
\beta_1&=&\epsilon_b \alpha_1+\beta_2-\epsilon_b\alpha_2.\label{eq:cand4a}
\eea
For a periodic candidate orbit we must again impose 
\be
\label{eq:cand_periodica}
(\alpha_0,\beta_0)\stackrel{!}{=}(\alpha_2,\beta_2).
\ee
Substituting \eqref{eq:cand_periodica} into \eqref{eq:cand1a}-\eqref{eq:cand4a} yields
a nonlinear system of four algebraic equations in four unknowns
$(\alpha_0,\beta_0,\alpha_1,\beta_1)$.

\begin{lem}
The system \eqref{eq:cand1a}-\eqref{eq:cand_periodica} can be simplified to a single
algebraic equation for $\beta_0$ given by
\bea
\label{eq:b0_equationa}
0&=&2(\beta_0-\xi)\left(\epsilon_b\mu-\alpha\xi+\alpha(\beta_0-\xi)/w_j(\beta_0)\right)\\
&&+\alpha\epsilon_b\ln\left(\frac{\xi(\mu\epsilon_b w_j(\beta_0)+
\alpha(\beta_0-\xi)\exp[\alpha(\xi-\beta_0)/\epsilon_b]}
{\beta_0(\mu\epsilon_b w_j(\beta_0)+\alpha(\beta_0-\xi))}\right)
=:W_j(\beta_0)\nonumber
\eea
where $w_j(\beta_0):=\exp[\alpha(\xi-\beta_0)/\epsilon_b]-1$.
\end{lem}

\begin{proof}
Similar steps as in the proof of Lemma \ref{lem:key_algebra1} are required. We have to replace the
algebraic equations \eqref{eq:cand1}-\eqref{eq:cand2} by \eqref{eq:cand1a}-\eqref{eq:cand2a}
and carry out lengthy, albeit quite direct, algebraic manipulations to obtain \eqref{eq:b0_equationa}.
\end{proof}

The next result follows from a direct calculation using \eqref{eq:b0_equationa} which we
omit here for brevity.

\begin{lem}
\label{lem:find_jump_jump}
The properties (P1)-(P5) from Lemma \ref{lem:find_canard_jump} hold verbatim if
$W_c$ is replaced by $W_j$.
\end{lem}

In fact, note that the arguments in Lemma \ref{lem:find_canard_jump} only depend upon
the sign of the derivative $W'_c(\xi)$ and the asymptotic behaviour of the logarithmic
summand in $W_c$. Lemma \ref{lem:find_jump_jump} states that the same technique can 
also be applied to $W_j$. From this last observation, the next two result follow
immediately.

\begin{cor}
\label{cor:sing_orbit1}
There exists an open set $(\mu_{j,1},\mu_{j,2})$ with $\mu_{j,1}<\mu_{j,2}$ such that the 
Olsen model \eqref{eq:Olsen2}, for $\mu\in(\mu_{j,1},\mu_{j,2})$ and otherwise standard 
parameter values from Table \ref{tab:tab1} ($k_1=0.41$), has a periodic 
candidate orbit $\psi_0$. It consists of two segments, one for the slow flow 
\eqref{eq:Olsen1_bu_sf_simple} without a canard segment and one consisting
of a large loop defined by \eqref{eq:Olsen2_sf_big}.
\end{cor}

\begin{cor}
\label{cor:delta_controls}
The open sets $(\mu_{c,1},\mu_{c,2})$ from Corollary \ref{cor:sing_orbit} and 
$(\mu_{j,1},\mu_{j,2})$ Corollary \ref{cor:sing_orbit1} have a non-empty intersection
{i.e.}~there exists an open set $(\mu_1,\mu_2)$ with $\mu_1<\mu_2$ such that
$(\mu_1,\mu_2)\subseteq (\mu_{j,1},\mu_{j,2})\cap (\mu_{c,1},\mu_{c,2})$.
\end{cor}

Essentially, Corollary \ref{cor:delta_controls} states that $\delta$ may deform
a canard-type orbit with maximal delay into an orbit which jumps near
the transcritical singularity; see also the discussion in Section \ref{sec:tr_res}
following equation \eqref{eq:delta_deform}. As before, we are not interested here in
any sharp quantitative bounds for $\mu_1$, $\mu_2$ in Corollary \ref{cor:delta_controls}.

\section{The Return Map}
\label{sec:retmap}

The last step is to construct the global return map using the result from
Sections \ref{sec:main_bu}-\ref{sec:delay} to obtain perturbation of the canard
candidate orbits constructed in Sections \ref{sec:loops}-\ref{sec:candidate}.
As in Section \ref{sec:candidate}, we are going to split the analysis of the 
return map into the two main cases from Theorem \ref{thm:main_intro}.

\subsection{The Canard Case}
\label{ssec:canard_case_map}

Before we can analyze the full return map several auxiliary results on the slow 
flow \eqref{eq:Olsen1_bu_sf_simple} are needed. Let 
$\phi_c:[a^*,\I)\times[b^*,\xi)\ra [a^*,\I)\times[\xi,\I)$ denote the slow flow 
map with maximal delay for an initial condition $(a,b)$ with $2ab<1$ and 
$b<\xi$ so that 
\benn
\phi_c(a,b)=\left(\frac{\mu}{\alpha}+e^{-\frac{2\alpha}{\epsilon_b}
\left(\xi-b\right)}\left(a-\frac{\mu}{\alpha}\right),2\xi-b\right).
\eenn
We assume that $\mu$ is chosen so that the candidate orbit constructed in 
Section \ref{ssec:canard_candidate} exists. Let $(\alpha_0,\beta_0)$ denote the landing 
point of this singular periodic orbit on $\cS_{2,0}^{a-}$ and, as before, let 
$(\alpha_1,\beta_1)=\phi_c(\alpha_0,\beta_0)$; see also Figure \ref{fig:fig6}.

\begin{lem}
\label{lem:ret1}
Let $\rho>0$ be sufficiently small then for $\phi_c(\alpha_0,\beta_0-\rho)=
(\alpha_{1*},\beta_{1*})$ we have $\beta_{1*}>\beta_1$ and for 
$\phi_c(\alpha_0,\beta_0+\rho)=(\alpha^*_{1},\beta^*_{1})$ we have 
$\beta^*_{1}<\beta_1$. 
\end{lem}

\begin{proof}
Under maximal delay we get $\beta_{1*}=2\xi-\beta_0+\rho>2\xi-\beta_0=\beta_1$ 
and for the second part $\beta^*_{1}=2\xi-\beta_0-\rho<2\xi-\beta_0=\beta_1$.
\end{proof}

\begin{figure}[htbp]
\psfrag{a}{$a$}
\psfrag{b}{$b$}
\psfrag{b=xi}{$\beta_2=\xi$}
\psfrag{2ab=1}{$2ab=1$}
\psfrag{a0b0+}{\scriptsize{$(\alpha_0,\beta_0+\rho)$}}
\psfrag{a0b0-}{\scriptsize{$(\alpha_0,\beta_0-\rho)$}}
\psfrag{a0b0}{\scriptsize{$(\alpha_0,\beta_0)$}}
\psfrag{a1b1+}{\scriptsize{$(\alpha_1^*,\beta_1^*)$}}
\psfrag{a1b1}{\scriptsize{$(\alpha_1,\beta_1)$}}
\psfrag{a1b1-}{\scriptsize{$(\alpha_{1*},\beta_{1*})$}}
\psfrag{a2b2+}{\scriptsize{$(\alpha_2^*,\beta_2^*)$}}
\psfrag{a2b2-}{\scriptsize{$(\alpha_{2*},\beta_{2*})$}}
\psfrag{slowflow}{slow flow}
\psfrag{line+}{\scriptsize{$b=\epsilon_b(a-\alpha_0)+\beta_0+\rho$}}
\psfrag{line}{\scriptsize{$b=\epsilon_b(a-\alpha_0)+\beta_0$}}
\psfrag{line-}{\scriptsize{$b=\epsilon_b(a-\alpha_0)+\beta_0-\rho$}}
\psfrag{alabel}{(a)}
\psfrag{blabel}{(b)}
	\centering
		\includegraphics[width=1\textwidth]{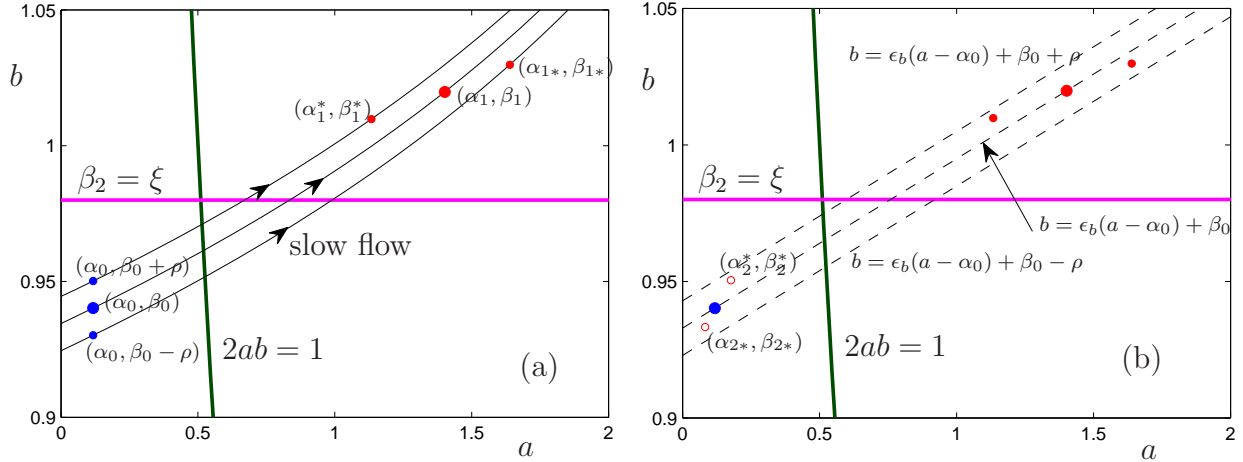}
	\caption{\label{fig:fig6}Numerical illustration of the results from Lemma \ref{lem:ret1} and
	Lemma \ref{lem:ret2}; slow flow map for the canard case. 
	Parameter values are $\epsilon_b=0.062$, $\kappa=3.93$, 
	$\xi=0.98$, $\alpha=0.37$, $\mu=1.3$. (a) Slow subsystem phase space with three 
	orbits (black curves) containing the three points (blue) $a=\alpha_0$, 
	$b=\beta_0-\rho,\beta_0,\beta_0+\rho$ for $\rho=0.01$ with $\alpha_0\approx 0.1176$ 
	and $\beta_0\approx 0.9402$. The image points under the slow flow map $\phi_c$ with 
	maximal delay (red) are shown as well. (b) Phase space with the three invariant lines 
	(dashed black, defined by \eqref{eq:inv_lines}). The thick points (blue/red) correspond 
	to the singular periodic orbit whereas the two circles (red) correspond to the images 
	under the global map defined by \eqref{eq:Olsen2_sf_big}.}
\end{figure}

We know that the canard candidate periodic orbit exists for $\beta_0<\xi$ under suitable 
conditions on $\mu$. The next result analyzes the slow dynamics of points near the candidate
orbit in more detail, which will be important for the stability of the periodic orbit.

\begin{lem}
\label{lem:ret2}
Under the assumptions of Lemma \ref{lem:ret1} we find that 
\benn
\beta_{1*}>\epsilon_b\alpha_{1*}+\beta_{0}-\rho-\epsilon_b\alpha_0 \qquad 
\text{and} \qquad \beta^*_{1}<\epsilon_b\alpha^*_{1}+\beta_0+\rho-\epsilon_b\alpha_0.
\eenn   
\end{lem}

\begin{proof}
Recall that the candidate orbit is given by the condition 
\benn
\beta_{1}=\epsilon_b\alpha_{1}+\beta_{0}-\epsilon_b\alpha_0.
\eenn
The slow flow on $\{x_2=0=y_2\}$ is two-dimensional, so trajectories cannot
intersect by uniqueness. Consider the slow flow trajectory $\gamma=\gamma(s)$ starting at 
$(\alpha_0,\beta_0-\rho)=\gamma(0)$. Observe that $\gamma(s)$ lies below the line 
given by $\beta=\epsilon_b\alpha_{1}+\beta_{0}-\rho-\epsilon_b\alpha_0$ for $0<s<\ll1$
and has to cross this line again so that it is close to the point $(\alpha_1,\beta_1)$;
see also Figure \ref{fig:fig6}. Note that we have used that $\rho$ is sufficiently small 
in the last step. The geometric crossing condition is equivalent to the algebraic 
condition
\benn
\beta_{1*}>\epsilon_b\alpha_{1*}+\beta_{0}-\rho-\epsilon_b\alpha_0
\eenn
as shown in Figure \ref{fig:fig6}(b). The second part is proven similarly, except that 
we notice that a trajectory starting 
at $(\alpha_0,\beta_0+\rho)$ must lie below the line 
$\beta=\epsilon_b\alpha_{1}+\beta_{0}+\rho-\epsilon_b\alpha_0$ when it reaches a 
neighbourhood of $(\alpha_1,\beta_1)$.
\end{proof}

Finally, we can proceed to prove the first part of the main result.

\begin{proof}(of Theorem \ref{thm:main_intro}, canard case)
The existence of the candidate $\psi_0$ is just a consequence of Corollary \ref{cor:sing_orbit}. 
To analyze the perturbation $\epsilon\in(0,\epsilon_0]$, we have to consider the global 
Poincar\'{e} map near the candidate orbit $\psi_0$. Fix a suitable small $\rho>0$ and 
define a cross-section
\benn
\Sigma_0:=\{(a,b,x,y)\in\bar{\cD}:a=\alpha_0+\rho,
b\in[\beta_0-\rho,\beta_0+\rho],x\in[0,\rho],y\in[0,\rho]\},
\eenn 
which is transverse to the flow on $\cS_{2,\epsilon}^{a-}$; the existence of such a section
follows from Proposition \ref{prop:second_chart_prop}, Fenichel theory and the transversality 
of the slow flow on $\cS_{2,0}^{a-}$ to $\{a=\alpha_0,b\in [\beta_0-\rho,\beta_0+\rho]\}$. 
Define another section
\benn
\Sigma_1:=\left\{(a,b,x,y)\in\bar{\cD}:a\in[\alpha_1-\rho,\alpha_1+\rho],
b\in[\beta_1-\rho,\beta_1+\rho],x=k\epsilon,y=\frac{x^2}{3ab}\right\},
\eenn   
where $k>0$ is a suitable constant. The flow induced map $\phi_{01}:\Sigma_0\ra\Sigma_1$ is 
a diffeomorphism due the canard case from Proposition \ref{prop:tc_canards} and 
since the exit from $\cS_{2,\epsilon}^{r+}$ is described by the center flow in chart 
$\kappa_1$ in Proposition \ref{prop:approach}(C4). For the global returns consider the section
\benn
\Sigma_2:=\left\{(a,b,x,y)\in\bar{\cD}:a\in[\alpha_0-\rho,\alpha_0+\rho],
b\in[\beta_0-\rho,\beta_0+\rho],x=k\epsilon,y=\frac{x^2}{3ab}\right\}.
\eenn 
The flow induced map $\phi_{01}:\Sigma_1\ra\Sigma_2$ is a diffeomorphism 
by Fenichel's Theorem applied to $\cC_0$. The global flow is approximated by the 
flow on attracting slow manifold $\cC_\epsilon$ which makes Proposition 
\ref{prop:large} applicable. Since $2\alpha_0\beta_0<1$ and $\beta_0<\xi$, it 
follows from Proposition \ref{prop:approach}(C1) that $\phi_{2,0}:\Sigma_2\ra \Sigma_0$ 
is a diffeomorphism defined via trajectories following the dynamics of the 
center-stable manifold $\cM_1$ in the chart $\kappa_1$. Note that $\Sigma_0$ is 
slightly shifted with $a=\alpha_0+\rho$ from the base point of the candidate orbit 
to avoid that points of the global large loops land exactly on the section. The global 
return map
\be
\label{eq:global_map_final}
\phi=\phi_{20}\circ\phi_{12}\circ\phi_{01}:\Sigma_0\ra\Sigma_0
\ee 
is exponentially contracting in the $(x,y)$-directions since (I) $\cC_\epsilon$ for 
$(x,y)$ bounded away from $(0,0)$ is attracting, (II) trajectories follow 
$S_{2,0}^{a-}$ and $S_{2,0}^{r+}$ in the chart $\kappa_2$ and (III) trajectories connect 
to $\cC_\epsilon$ in the entrance and exit chart $\kappa_1$. If we can show that the map 
$\phi$ also contracts along the $b$-direction the result will follow.

The contraction in the $b$-direction can be derived by using Lemmas \ref{lem:ret1}-\ref{lem:ret2}. 
Indeed, consider first the point $(\alpha_0,\beta_0-\rho)$ then by Lemma \ref{lem:ret2} the 
image $(\alpha_{1*},\beta_{1*})=\phi_c(\alpha_0,\beta_0-\rho)$ lies between the lines which 
are invariants for $(\alpha_0,\beta_0)$ and $(\alpha_0,\beta_0-\rho)$ for the global flow 
from Proposition \ref{prop:large}; see also Figure \ref{fig:fig6}. By Lemma \ref{lem:returns} 
the global return of $(\alpha_{1*},\beta_{1*})$ governed by the flow on $\cC_0$ ends at a point 
$(\alpha_{2*},\beta_{2*})$ with $\alpha_{2*}<\alpha_0$ which again lies between the same two lines. 
The slow flow from $(\alpha_{2*},\beta_{2*})$ back to a section 
$\{a=\alpha_0,b\in[\beta_0-\rho,\beta_0]\}$ does not change this property since the $b$-coordinate 
on $\cS_{2,\epsilon}^{a-}$ increases. 

For sufficiently small $\rho$ the same argument applies for a point 
$(\alpha_0+\rho,\beta_0-\rho)\in\Sigma_0$. Indeed, as for $(\alpha_0,\beta_0-\rho)$ one may 
consider a point $(\alpha_0,\beta_0+\rho)$ with the minor modification that we start with 
a point in the interior of the open set between $(\alpha_0,\beta_0)$ and $(\alpha_0,\beta_0+\rho)$ 
in $\Sigma_0$ and argue in backward-time {i.e.}~points with fixed $\alpha_0$ lying above 
$(\alpha_0,\beta_0)$ move away from the $\beta_0$ in backward time. The same applies for 
points with $\alpha_0+\rho$ lying on $\Sigma_0$.

Hence the full map $\phi$ also contracts along the $b$-direction. The existence of an 
attracting fixed point now follows, {e.g.}~from the Banach fixed point theorem. This fixed
point is precisely the intersection of an orbit $\psi_\epsilon$ with $\Sigma_0$.
\end{proof}

In the proof we could have inserted another section between $\Sigma_0$ and $\Sigma_1$ to 
describe the exit to $\cC_\epsilon$ via Proposition \ref{prop:approach}(C4) separately. 
Alternatively, we could also have removed $\Sigma_2$ and treated the transition map from 
$\Sigma_1$ to $\Sigma_0$ at once. 

\subsection{The Jump Case}
\label{ssec:jump_case_map}

Although a similar argument as for the canard case can be followed, we have to 
replace Lemma \ref{lem:ret1} and \ref{lem:ret2}. Consider the slow 
flow \eqref{eq:Olsen1_bu_sf_simple} and let 
$\phi_j:[a^*,\I)\times[b^*,\xi)\ra [a^*,\I)\times[\xi,\I)$ denote the slow flow 
map for the jump case for an initial condition $(a,b)$ with $2ab<1$ and 
$b<\xi$ so that 
\benn
\phi_j(a,b)=\left(\frac{\mu}{\alpha}+e^{-\frac{\alpha}{\epsilon_b}
\left(\xi-b\right)}\left(a-\frac{\mu}{\alpha}\right),\xi\right).
\eenn
We assume that $\mu$ is chosen so that the candidate orbit constructed in 
Section \ref{ssec:jump_candidate} exists. Let $(\alpha_0,\beta_0)$ denote the landing 
point of this singular periodic orbit on $\cS_{2,0}^{a-}$ and, as before, let 
$(\alpha_1,\beta_1)=\phi_j(\alpha_0,\beta_0)$; see also Figure \ref{fig:fig10}. Note that
we always have $\beta_1=\xi$ for the jump case and hence we do not have to control
the $b$-coordinate.

\begin{figure}[htbp]
\psfrag{a}{$a$}
\psfrag{b}{$b$}
\psfrag{b=xi}{$\beta_2=\xi$}
\psfrag{2ab=1}{$2ab=1$}
\psfrag{a0b0+}{\scriptsize{$(\alpha_0,\beta_0+\rho)$}}
\psfrag{a0b0-}{\scriptsize{$(\alpha_0,\beta_0-\rho)$}}
\psfrag{a0b0}{\scriptsize{$(\alpha_0,\beta_0)$}}
\psfrag{a1b1+}{\scriptsize{$(\alpha_1^*,\beta_1^*)$}}
\psfrag{a1b1}{\scriptsize{$(\alpha_1,\beta_1)$}}
\psfrag{a1b1-}{\scriptsize{$(\alpha_{1*},\beta_{1*})$}}
\psfrag{a2b2+}{\scriptsize{$(\alpha_2^*,\beta_2^*)$}}
\psfrag{a2b2-}{\scriptsize{$(\alpha_{2*},\beta_{2*})$}}
\psfrag{slowflow}{slow flow}
\psfrag{line+}{\scriptsize{$b=\epsilon_b(a-\alpha_0)+\beta_0+\rho$}}
\psfrag{line}{\scriptsize{$b=\epsilon_b(a-\alpha_0)+\beta_0$}}
\psfrag{line-}{\scriptsize{$b=\epsilon_b(a-\alpha_0)+\beta_0-\rho$}}
\psfrag{alabel}{(a)}
\psfrag{blabel}{(b)}
	\centering
		\includegraphics[width=1\textwidth]{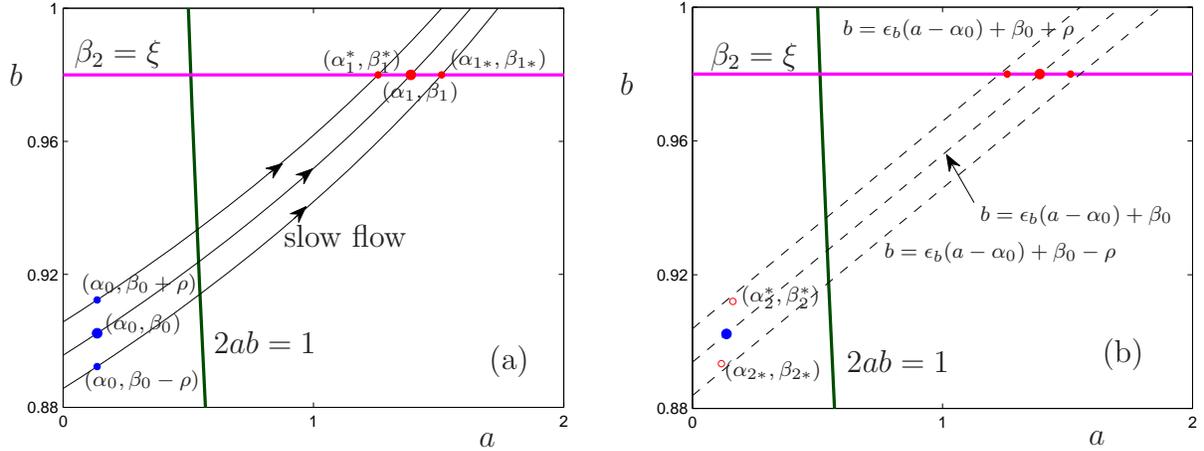}
	\caption{\label{fig:fig10}Numerical illustration of the results from Lemma \ref{lem:ret4}; 
	slow flow map for the jump case. 
	Parameter values are $\epsilon_b=0.062$, $\kappa=3.93$, 
	$\xi=0.98$, $\alpha=0.37$, $\mu=1.3$. (a) Slow subsystem phase space with three 
	orbits (black curves) containing the three points (blue) $a=\alpha_0$, 
	$b=\beta_0-\rho,\beta_0,\beta_0+\rho$ for $\rho=0.01$ with $\alpha_0\approx 0.1362$ 
	and $\beta_0\approx 0.9023$. The image points under the slow flow map $\phi_j$ 
	(red) are shown as well. (b) Phase space with the three invariant lines 
	(dashed black, defined by \eqref{eq:inv_lines}). The thick points (blue/red) correspond 
	to the singular periodic orbit whereas the two circles (red) correspond to the images 
	under the global map defined by \eqref{eq:Olsen2_sf_big}.}
\end{figure}

\begin{lem}
\label{lem:ret4}
Let $\rho>0$ be sufficiently small and let $\phi_j(\alpha_0,\beta_0-\rho)=
(\alpha_{1*},\xi)$ and 
$\phi_c(\alpha_0,\beta_0+\rho)=(\alpha^*_{1},\xi)$. Then we have
\benn
\frac{1}{\epsilon_b}(\xi-\beta_{0}+\rho+\epsilon_b\alpha_0)>\alpha_{1*}
\qquad 
\text{and}
\qquad
\frac{1}{\epsilon_b}(\xi-\beta_{0}-\rho+\epsilon_b\alpha_0)<\alpha^{*}_1.
\eenn  
\end{lem}

\begin{proof}
The proof follows the same idea as in Lemma \ref{lem:ret2} {i.e.} a continuity 
argument for small $\rho$ and the standard uniqueness result for ODEs applied
to the planar slow flow on $S_{2,0}^{a-}$; see also Figure \ref{fig:fig10}(b).
\end{proof}

Now we may finish the proof for the second part of the main result.

\begin{proof}(of Theorem \ref{thm:main_intro}, jump case and final result)
The same steps as in the proof of the canard case in Section 
\ref{ssec:canard_case_map} can be applied upon noticing the 
following aspects:
\begin{itemize}
 \item We always have $\beta_1=\xi$.
 \item Applying the case (C4) from Proposition \ref{prop:approach} is 
still valid due to Corollary \ref{cor:jump_canard_chart1}.
 \item Instead of the canard case, we have to apply the jump case of Proposition \ref{prop:tc_canards}.
 \item The $b$-direction contraction from Lemma \ref{lem:ret2} is replaced by Lemma \ref{lem:ret4}.
\end{itemize}
As the remaining elements of the jump case proof are similar, we do 
not provide the details here. To conclude that there is indeed an open set 
of $\mu$-values for the which the canard and jump case can be obtained, just via 
a variation of $\delta$, we may apply Corollary \ref{cor:sing_orbit}.
\end{proof}

\section{Outlook}
\label{sec:outlook}

In addition to the non-classical relaxation oscillations described in Theorem 
\ref{thm:main_intro}, there are several other dynamical regimes of interest in
the Olsen model. We do not provide the full details here and just give a
brief geometric description of the other two cases observed by Olsen as shown 
in Figure \ref{fig:fig1}.\medskip

\begin{figure}[htbp]
\psfrag{af}{\scriptsize{(a)}}
\psfrag{bf}{\scriptsize{(b)}}
\psfrag{ab}{\scriptsize{$a,b$}}
\psfrag{a}{\scriptsize{$a_2$}}
\psfrag{b}{\scriptsize{$b_2$}}
\psfrag{x}{\scriptsize{$x$}}
\psfrag{y}{\scriptsize{$y$}}
\psfrag{xy}{\scriptsize{$x_2,y_2$}}
\psfrag{L0}{\scriptsize{$\cL_0$}}
\psfrag{C0}{\scriptsize{$C_0$}}
\psfrag{gm}{\scriptsize{$\gamma_m$}}
\psfrag{xy0}{\scriptsize{$\{x_2=0=y_2\}$}}
\psfrag{bxi}{\scriptsize{$b_2=\xi$}}
\psfrag{Ca1}{\scriptsize{$\cU_0^a$}}
\psfrag{Cr1}{\scriptsize{$\cU_0^r$}}
\psfrag{Ca2}{\scriptsize{$\cV_0^a$}}
\psfrag{Cr2}{\scriptsize{$\cV_0^r$}}
	\centering
		\includegraphics[width=1\textwidth]{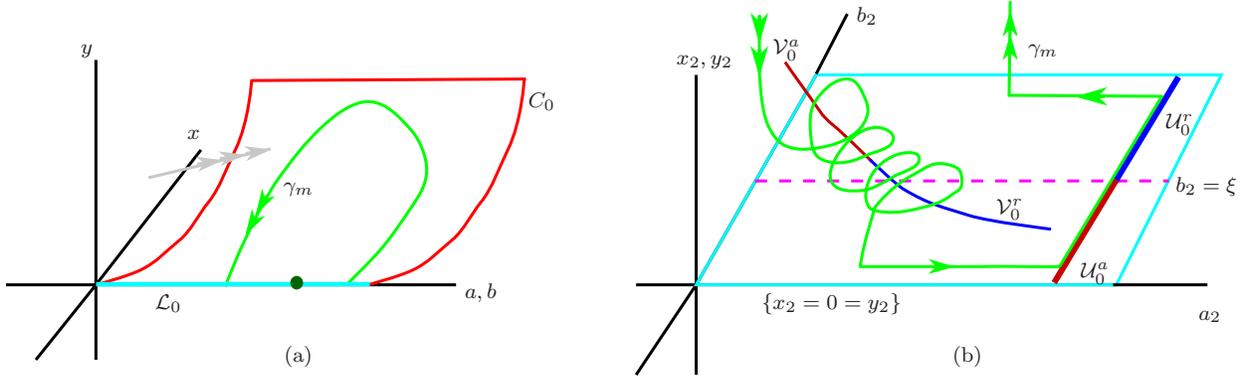}
	\caption{\label{fig:fig9}Sketch of the basic geometry for MMOs 
	inside the region $\cD$ under the assumption $\epsilon_b\ra 0$. (a) Phase space for
	the system \eqref{eq:Olsen2} which captures the large fast loops. The critical 
	manifold $C_0$ (red), two segments of an MMO candidate orbit $\gamma_m$ (green),
	the fold manifold $\cL_0$ (cyan), the submanifold $\{2ab=1,x=0=y\}$ (dark-green dot) and
	the ``super-fast'' attracting dynamics (grey triple arrow) are shown. (b) Phase space
	for \eqref{eq:Olsen2} with $\epsilon_b\ra 0$. The focus is on the slow drift near $\cL_0$ 
	(cyan) and the ``super-slow'' dynamics near the critical manifold $\cU_0=\{a_2=\mu/\alpha\}$
	of \eqref{eq:Olsen1_bu_sf_simple}. The candidate orbit $\gamma_m$ (green), the 
	exchange-of-stability line $\{b_2=\xi\}$ (magenta) and the one-dimensional critical 
	manifold $\cV_0$ of \eqref{eq:Olsen1} are shown; note that 
	$\cV_0\cap\{x_2=0=y_0\}=\emptyset$ {i.e.}~$\cV_0$ lies entirely above 
	the submanifold $\{x_2=0=y_0\}$. Furthermore, the two one-dimensional parts of the
	critical manifold split attracting (dark red) and repelling (blue) parts as 
	$\cU_0=\cU_0^a\cup\cU_0^r$ and $\cV_0=\cV_0^a\cup\cV_0^r$. For a description of
	the dynamics please refer to the text in Section \ref{sec:outlook}.}
\end{figure}

We start with the case of MMOs. Part of the basic idea how MMOs may be generated
can be found in \cite{DesrochesKrauskopfOsinga1}. However, with the results developed
in this paper, we can already give a substantially more detailed description. 

First, we observe that $k_1=0.16$ corresponds to the case in Table \eqref{sec:tr_res}, where 
$\epsilon_b$ is also a small parameter. On a formal level, we still start with the system 
\eqref{eq:Olsen2}, and note that the reasonable assumption $\epsilon_b\epsilon\ll\epsilon_b$
implies that $x$ is still the fastest variable and we may reduce the situation to a ``slow'' 
vector field on the normally hyperbolic part of $\cC_0$. This vector field is still solvable 
explicitly with $0<\epsilon_b\ll1$ as discussed in Section \ref{sec:loops}; see also Figure 
\ref{fig:fig9}(a). The blow-up analysis 
in Section \eqref{sec:main_bu} has to be re-considered as we have to append $\epsilon_b'=0$. 
Let us assume, {i.e.}~we do not prove this conjecture here, that the main dynamical generating
mechanism for the slowest dynamics is governed by the system \eqref{eq:Olsen1}, where $\epsilon_b$
is now another small parameter. Then \eqref{eq:Olsen1} can be viewed as a fast-slow system with
3 fast variables and 1 slow variable. The critical manifold for this system is given by solving
the algebraic equations  
\be
\label{eq:Olsen_outlook}
\begin{array}{rcl}
0&=& \mu-\alpha a_2 -a_2b_2y_2,\\
0&=& b_2x_2-x_2^2 +3a_2b_2y_2-\xi x_2,\\
0&=&x_2^2-y_2-a_2b_2y_2,\\
\end{array}
\ee 
where we have assumed that $\delta=0$ for convenience. The critical manifold described by 
\eqref{eq:Olsen_outlook} is given by one-dimensional curves. One part is given by
\benn
\cV_0:=\left\{x_2=0=y_2,a_2=\frac{\mu}{\alpha}\right\}=\cV_0^a\cup p_{\cV}\cup \cV_0^r,
\eenn 
where $p_{\cV}=\{b_2=\xi,x_2=0=y_2,a_2=\frac{\mu}{\alpha}\}$, $\cV_0^a=\cV_0\cap\{b_2<\xi\}$
and $\cV_0^r=\cV_0\cap\{b_2>\xi\}$; see also Figure \ref{fig:fig9}(b). $\cV_0^a$ is normally 
hyperbolic attracting, $\cV_0^r$ is normally hyperbolic repelling and $p_{\cV}$ is not 
normally hyperbolic. Furthermore, there exists another curve
\benn
\cU_0=\cU_0^a\cup p_{\cU}\cup \cU_0^r\subset \cS^{a+}_{2,0}\cup\{b_2=\xi,2a_2\xi=1\}\cup 
\cS^{r-}_{2,0}
\eenn
as shown in Figure \ref{fig:fig9}(b), where $\cS^{a+}_{2,0}$ and $\cS^{r-}_{2,0}$ are the 
two-dimensional critical manifolds illustrated and discussed in Section \ref{ssec:chart2} and
illustrated in Figure \ref{fig:fig5}. In particular, $\cU_0$ also consists of three parts
where one may check that $\cU_0^a\subset \{b_2>\xi\}$ is normally hyperbolic attracting 
with a linearization of the the fast subsystem having a real negative eigenvalue and a pair of complex 
conjugate eigenvalues with negative real parts. $\cU_0^r\subset \{b_2<\xi\}$ is normally 
hyperbolic repelling with a linearization of the the fast subsystem having a real negative eigenvalue
and a pair of complex conjugate eigenvalues with positive real parts. A (delayed) Hopf 
bifurcation \cite{Neishtadt1,Neishtadt2} occurs at $p_{\cU}$. This mechanism generates 
SAOs via a tourbillon-type mechanism \cite{Desrochesetal} as trajectories spiral around 
$\cU_0$; see also Figure \ref{fig:fig9}(b). More precisely, after a large loop, trajectories 
spiral towards $\cU_0^a$, including a slow drift towards $p_{\cU}$. After the delayed 
Hopf bifurcation, trajectories spiral outwards around $\cU_0^r$.

Then, we note that $\{x_2=0=y_2\}$ is still invariant. Since $\epsilon_b$ is now viewed
as a singular perturbation parameter, we can try to approximate the transition near $\{x_2=0=y_2\}$
towards $\cV_{0}^a$ via the one-dimensional system
\benn
\frac{da_2}{ds}= \mu-\alpha a_2, 
\eenn
which is just \eqref{eq:Olsen1_bu_sf_simple} for $\epsilon_b=0$. Trajectories reach a 
neighbourhood of $\cV_{0}^a$ and then drift slowly towards $p_{\cV}$. One has to prove 
an analogous result to the transcritical passage in Section \ref{sec:delay} near $p_{\cV}$.
Trajectories eventually leave $\cV_0^r$ and start another large loop as shown in 
Figure \ref{fig:fig9}(b). A periodic orbit corresponds to an MMO as shown in Figure 
\ref{fig:fig1}(a).

Although the description of MMOs we have just given is clearly not rigorous, the 
geometric structure suggested by Figure \ref{fig:fig9} indicates that a similar strategy
as carried out in Sections \ref{sec:main_bu}-\ref{sec:retmap} could work to prove the
existence of MMOs. This problem is left open and could be considered in future work.\medskip          

\begin{figure}[htbp]
\psfrag{af}{\scriptsize{(a)}}
\psfrag{bf}{\scriptsize{(b)}}
\psfrag{a}{\scriptsize{$a_2$}}
\psfrag{b}{\scriptsize{$b_2$}}
\psfrag{osc}{\scriptsize{oscillation}}
\psfrag{slide}{\scriptsize{sliding}}
\psfrag{xy}{\scriptsize{$x_2,y_2$}}
\psfrag{xy0}{\scriptsize{$\{x_2=0=y_2\}$}}
\psfrag{bxi}{\scriptsize{$b_2=\xi$}}
	\centering
		\includegraphics[width=1\textwidth]{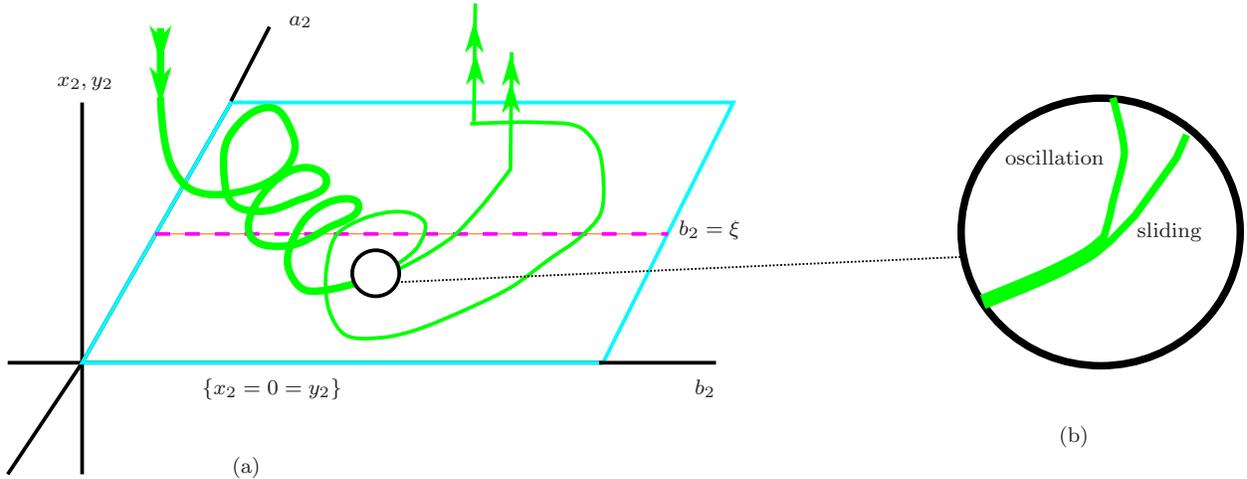}
	\caption{\label{fig:fig11}Sketch of the geometry for a possible chaos generating 
	mechanism in the Olsen model. (a) Phase space
	for \eqref{eq:Olsen2} in the intermediate regime between the non-classical 
	relaxation oscillation case in Figure \ref{fig:fig8} and the MMO case in
	Figure \ref{fig:fig9}. The focus is on the slow drift near $\cL_0$ 
	(cyan). (b) Sketch of the mechanism which causes the strong stretching of trajectories. 
	One part of the orbits tend to make one more oscillation similar to the MMO case. The second
	part starts to ``slide'' on the set $\{x_2=0=y_2\}$ similar to the non-classical relaxation 
	oscillation case. For a more detailed description of the dynamics please refer to the text 
	in Section \ref{sec:outlook}.}
\end{figure}

The time series in Figure \ref{fig:fig1} suggest that there is an intermediate case between
the regime of MMOs and the non-classical relaxation oscillations, where the Olsen model 
is chaotic; in particular, see Figure \ref{fig:fig1}(b).

Several chaos-generating mechanism have been identified for fast-slow systems. We briefly recall
two cases for classical relaxation oscillation system in $\R^3$. One possibility is that jumps from
a fold curve land on an attracting slow manifold where the slow flow has a tangency to the projection
of the fold curve along the fast direction \cite{MKKR}. It has been shown in 
\cite{GuckenheimerWechselbergerYoung} that there exists a near one-dimensional return map
which is similar to a H\'{e}non-type map. However, the basic mechanism for the flow to generate sensitive 
dependence upon initial conditions is that near the tangency orbits are ``split'' 
into different directions. Another possible chaos-generating mechanism has been identified
in \cite{Haiduc1} based upon canard orbits arising from a folded saddle. In this context, there 
is also a ``splitting''-type mechanism. Orbits follow the same canard but jump into different
directions when departing from it, as well as departing from the fold curve where the fold saddle
is based. 

For both mechanisms, there is a region of phase space, where orbits are drastically separated.
In combination with a global return mechanism, one obtains the main ingredients (stretching and folding)
for Smale horseshoe dynamics \cite{GuckenheimerWechselbergerYoung,Haiduc1}.

The Olsen model seems to exhibit a different mechanism, which also induces the drastic separation
of orbits in part of the phase space. Figure \ref{fig:fig11} provides a basic sketch of the 
mechanism which we conjecture. Consider the singular limit $\epsilon=0=\epsilon_b$. Then 
there exists a family of fast subsystem periodic orbits around the repelling critical manifold 
$\cV_0^r$ which may become tangent to the invariant submanifold $\{x_2=0=y_2\}$. If the system 
operates in a parameter regime between MMOs and non-classical relaxation oscillations, it 
could happen that a bundle of trajectories spirals in a region where $\cV_0^r$ (and 
$\cS_{2,0}^{r-}$) are located. During their last SAO before reaching a neighbourhood of 
$\{x_2=0=y_2\}$ some orbits may make one additional SAO, while others will tend to ``slide'' immediately
towards $\{b_2=\xi\}$; see Figure \ref{fig:fig11}. This effect may cause the separation effect 
required to obtain a Smale horseshoe. The global returns are 
still controlled via $\cC_0$. Let us note that this mechanism also shares some 
similarities with grazing-sliding bifurcations of periodic orbits discussed recently in the 
context of non-smooth dynamical systems \cite{PSDS}. We are going to make this relation, and the chaos 
generating mechanism itself, more precise in future work.\medskip

In summary, we have given a precise description of the geometric structure and local asymptotic 
stability of non-classical relaxation oscillations originally discovered by Olsen more than 
30 years ago. In particular, we have seen how the interaction between the two perturbation 
parameters $\epsilon$ and $\delta$ can be exploited to provide a coherent picture for the original 
numerical simulations as well as numerical continuation results obtained in \cite{DesrochesKrauskopfOsinga1} 
for a case between $\delta=0$ and $\delta=K_2\epsilon^2$. Furthermore, we have briefly outlined how the 
analysis could be continued to cover the MMO and chaotic cases. The difficulty of the analysis shows 
that multiple time scale systems, which are not in standard form and possess several singular 
perturbation parameters, provide an interesting challenge for geometric singular perturbation theory. 
\medskip   

\textbf{Acknowledgments:} CK would like to thank the Austrian Academy of Sciences ({\"{O}AW}) 
for support via an APART fellowship. CK and PS would like to thank the European Commission
(EC/REA) for support by a Marie-Curie International Re-integration Grant.

\newpage

\appendix

\section{Normally Hyperbolicity \& Fast-Slow Systems}
\label{ap:fastslow}

We only recall the basic definitions and results about fast-slow systems. There are several 
standard references that detail many parts of the theory 
\cite{Jones,KaperJonesPrimer,MisRoz,Desrochesetal,ArnoldEncy}. A fast-slow system of ordinary 
differential equations (ODEs) is given by:
\be
\label{eq:basic1}
\begin{array}{rcrcl}
\epsilon \dot{x}&=&\epsilon\frac{dx}{d\tau}&=&f(x,y,\epsilon),\\
\dot{y}&=&\frac{dy}{d\tau}&=&g(x,y,\epsilon),\\
\end{array}
\ee 
where $x\in\R^m$ are fast variables, $y\in\R^n$ are slow variables and $0<\epsilon\ll1$ is a 
small parameter representing the ratio of time scales. The maps $f,g$ are assumed to be 
sufficiently smooth. Equation \eqref{eq:basic1} can be 
re-written by changing from the slow time scale $\tau$ to the fast time scale $t=\tau/\epsilon$
\be
\label{eq:basic2}
\begin{array}{lclcr}
x'&=&\frac{dx}{dt}=f(x,y,\epsilon),\\
y'&=&\frac{dy}{dt}=\epsilon ~g(x,y,\epsilon).\\
\end{array}
\ee
The singular limit $\epsilon\ra 0$ of \eqref{eq:basic2} yields the fast subsystem 
ODEs parametrized by the slow variables $y$. Setting $\epsilon\ra 0$ in \eqref{eq:basic1} 
gives a differential-algebraic equation (DAE), called the slow subsystem, on the 
critical manifold $\cC_0:=\{(x,y)\in\R^{m+n}:f(x,y,\epsilon)=0\}$. Concatenations of 
fast and slow subsystem trajectories are called candidates \cite{Benoit2,Haiduc1}. 

A subset $\cS\subset \cC$ is called normally hyperbolic if the $m\times m$ total derivative 
matrix $(D_xf)(p)$ is hyperbolic for $p\in \cS$. A normally hyperbolic subset $\cS$ is attracting 
if all eigenvalues of $(D_xf)(p)$ have negative real parts for $p\in \cS$, $\cS$ is called 
repelling if all eigenvalues have positive real parts and of saddle-type if there are 
positive and negative eigenvalues. On normally hyperbolic parts of $\cC$ the implicit function 
theorem applies to $f(x,y,0)=0$ providing a map $h_0(y)=x$ so that $\cC$ can be expressed (locally) 
as a graph. 

\begin{thm}[Fenichel's Theorem \cite{Fenichel4,Jones,Tikhonov}]
\label{thm:fenichel1}
Suppose $\cS=\cS_0$ is a compact normally hyperbolic submanifold (possibly with boundary) of the 
critical manifold $\cC_0$. Then, for $\epsilon>0$ sufficiently small, there exists a locally 
invariant manifold $\cS_\epsilon$ diffeomorphic to $\cS_0$. $\cS_\epsilon$ has a distance of 
$\cO(\epsilon)$ from $\cS_0$ and the flow on $\cS_\epsilon$ converges to the slow flow as 
$\epsilon \to 0$.
\end{thm}

The distance between $\cS_\epsilon$ and $\cS_0$ can be expressed in the Hausdorff metric or a
suitable $C^r$-norm (using the map $h_0$ and its perturbation $h_\epsilon$). A manifold 
$\cS_\epsilon$ provided by Fenichel's Theorem is called a slow manifold. Slow manifolds 
are usually not unique but different slow manifolds lie at a distance $\cO(e^{-K/\epsilon})$ 
for some constant $K>0$. Often we shall we shall make a choice of compact subset and choice 
of slow manifold without further notice, indicating that the choice does not matter for the 
asymptotic analysis performed.
 
A trajectory is called a maximal canard if it lies in the intersection of an attracting and 
a repelling slow manifold. Canards were first investigated by a group of French mathematicians 
\cite{BenoitCallotDienerDiener} using nonstandard analysis. Later also asymptotic 
\cite{Eckhaus,BaerErneux1} and geometric \cite{DumortierRoussarie,KruSzm3} methods have been 
developed to understand canard orbits.

\section{Geometric Desingularization}
\label{ap:blowup}

Here we shall briefly review the basic strategy for the blow-up approach for geometric 
desingularization of fast-slow systems. Details on the classical, single-scale, method
can be found {e.g.}~in \cite{Dumortier1}. The classical blow-up was first introduced into
fast-slow systems in \cite{DumortierRoussarie}. Further developments can be found in 
\cite{KruSzm3}; see also the introduction in \cite{KruSzm1}.\medskip

The starting point is to write the system \eqref{eq:basic2} as follows
\be
\label{eq:bu_start}
\begin{array}{lcl}
x'&=&f(x,y,\epsilon),\\
y'&=&\epsilon~ g(x,y,\epsilon),\\
\epsilon'&=&0.\\
\end{array}
\ee
Let us denote the vector field defined by \eqref{eq:bu_start} as $X$ {i.e.}~$X$ is a 
mapping 
\benn
X:\R^{m+n}\times [0,\epsilon_0)\ra T\left(\R^{m+n}\times [0,\epsilon_0)\right)
\eenn
where $T(\cdot)$ indicates the tangent bundle. Further equations for parameters could 
be appended to \eqref{eq:bu_start} as well, if necessary. Suppose 
\eqref{eq:bu_start} has an equilibrium point for $\epsilon=0$, or more generally a 
submanifold $\cM=\{f=0\}$ of equilibria in $\R^{m+n}\times \{\epsilon=0\}$. If 
$(D_xf)(p)$ is not a hyperbolic matrix for each $p\in\cM$, the equilibrium (manifold)
$\cM$ is degenerate and classical linearization results do not apply directly to 
\eqref{eq:bu_start}. 

The blow-up technique is based upon replacing $\cM$ by a, usually more complicated,
manifold $\bar{\cM}$ and using a map
\benn
\Phi:\bar{\cM}\ra \cM
\eenn
which induces a vector field $\bar{X}$ on $\bar{\cM}$ via the pushforward $\Phi_*$
and the condition $\Phi_*(\bar{X})=X$. Using a good choice for $\bar{\cM}$,
one may often analyze the blown-up vector field $\bar{X}$ as it is possible that
invariant manifolds of $\bar{X}$ are now (partially) hyperbolic. 

As an example, consider the classical case when $(x,y)\in\R^2$ and $f(x,y)=y-x^2$ 
is the (truncated) normal form of a fold bifurcation. The origin $(x,y,\epsilon)=(0,0,0)$ 
is the important non-hyperbolic point and the standard choice is to use a sphere
for geometric desingularization $\bar{\cM}:=S^2\times [0,r_0)$ for some constant 
$r_0>0$ or $r_0=+\I$. Therefore, one has essentially inserted a sphere at the origin;
see also \cite{DumortierRoussarie,KruSzm1}. 

Although one could try to find a suitable global parametrization of $\bar{\cM}$, this
is usually not very convenient for calculations. Instead, one uses charts 
$\kappa_j:\bar{\cM}\ra \R^{m+n+1}$ of $\bar{\cM}$ for the calculations,
 which is illustrated by the following important diagram
\benn
\xymatrix{& \bar{M} \ar[ld]_{\kappa_j} \ar[rd]^{\Phi} &  \\
 \R^{m+n+1} \ar[rr]^{\varphi_j}  &  & \R^{m+n+1}  }
\eenn
which commutes. Hence, one may just try to calculate the map $\varphi_j$ and obtain
a vector field on $\R^{m+n+1}$ by applying the coordinate change 
\benn
(x_j,y_j,\epsilon_j)=\varphi^{-1}(x,y,\epsilon).
\eenn
One may often, via a good choice of $\bar{\cM}$ and chart maps $\kappa_j$, compute the
vector fields in $(x_j,y_j,\epsilon_j)$-coordinates. The same remark applies to the
transition maps between different charts $\kappa_{jk}$. Section \ref{sec:main_bu} 
carries out these calculations for a submanifold of fold points in the Olsen model. 
 
\section{An Auxiliary Center Manifold Reduction}
\label{ap:cm1}

Here we present the details for the center manifold calculation for \eqref{eq:bu_kappa1_vf_cm3D}. 
We drop the sub- and superscripts of $(r_1,y_1,\epsilon_1)$ and $(\alpha_1^*,b_1^*)$ for notational 
convenience; all variables and constants used in this section are temporary and should not be 
confused with notation within the main manuscript. Re-ordering the variables and translating 
\eqref{eq:bu_kappa1_vf_cm3D} via $Y=y-1/(3ab)$ yields   
\be
\label{eq:bu_kappa1_vf_cm3D_short}
\begin{array}{lcl}
r'&=& r\left[\epsilon(b-\xi)+3abY\right]=:f_1(r,\epsilon,Y),\\
\epsilon'&=&-\epsilon\left[\epsilon(b-\xi)+3abY\right]=:f_2(r,\epsilon,Y)\\
Y'&=& f_3(r,\epsilon,Y),\\
\end{array}
\ee 
where the function $f_3$ is given by
\benn
f_3(r,\epsilon,Y):=\kappa\epsilon(1-[Y+1/(3ab)][1+a^*_1b^*_1])
-2[Y+1/(3ab)]\left(\epsilon(b-\xi)+3aby\right).
\eenn
Let $z:=(r,\epsilon,Y)^T$ and consider
\benn
A:=\left.D_z(z')\right|_{(0,0,0)}=
\left(
\begin{array}{ccc}
0 & 0 & 0 \\
0 & 0 & 0 \\
0 & K & -2 \\
\end{array}
\right)\qquad \text{and}\qquad M:=\left(
\begin{array}{ccc}
1 & 0 & 0 \\
0 & -2/K & 0 \\
0 & 1 & 1 \\
\end{array}
\right),
\eenn
where $K:=-(2 b + \kappa - 2 a b \kappa - 2 \xi)/(3 a b)$. Let 
$(x_1,x_2,\tilde{y})^T=\tilde{z}=M^{-1}z$ and observe that 
$M^{-1}AM=J\in\R^{3\times 3}$ with $J_{33}=-2$ and $J_{ij}=0$ otherwise. 
Set $\tilde{z}=(x_1,x_2,\tilde{y})=M^{-1}z$ so that
\be
\label{eq:sf_cm_theory}
\begin{array}{ccccc}
\left(\begin{array}{c}x_1'\\ x_2'\\\end{array}\right)&=&
\left(\begin{array}{cc}0 & 0\\ 0& 0\\\end{array}\right)
\left(\begin{array}{c}x_1\\ x_2\\\end{array}\right)&+& 
\left(\begin{array}{c}F_1(x_1,x_2,\tilde{y})\\ F_2(x_1,x_2,\tilde{y})\\\end{array}\right)\\
\tilde{y}'&=& -2\tilde{y}&+& G(x_1,x_2,\tilde{y}),\\ 
\end{array}                                                                                                                                        
\ee
where $(F_{1},F_2,G)^T=(0,0,2\tilde{y})^T+M^{-1}(f_1(M\tilde{z}),
f_2(M\tilde{z}),f_3(M\tilde{z}))^T$. The system \eqref{eq:sf_cm_theory} is in 
the standard form for center manifold theory \cite{GH}. The usual perturbation 
ansatz is $\tilde{y}=h(x_1,x_2)=k_{11}x_1^2+k_{12}x_1x_2+k_{22}x_2^2+\cO(3),$ 
where $\cO(3):=\cO(x_1^3,x_1^2x_2,x_1x_2^2,x_2^3)$. The defining invariance equation 
for the center manifold with $x=(x_1,x_2)^T$ and $F=(F_1,F_2)^T$ is  
\be
\label{eq:inv_eq_cm1}
Dh(x)F(x,h(x))=-2h(x)+G(x,h(x))
\ee
since the $x'$-equations in \eqref{eq:sf_cm_theory} have no linear term. Collecting 
terms of order $\cO(x_1^2)$ in \eqref{eq:inv_eq_cm1} gives $k_{11}=0$ and the 
$\cO(x_1x_2)$-terms give $k_{12}=0$. For $\cO(x_2^2)$ equation \eqref{eq:inv_eq_cm1} 
and $k_{11}=0=k_{12}$ imply
\benn
k_{22}=\frac{3ab(1+4ab)\kappa}{4(b-\xi)+2\kappa(1-2ab)}.
\eenn  
Transforming back to the variables $(r,\epsilon,y)$ via the matrix $M$ and translation 
yields the center manifold
\benn
y=\frac{1}{3ab}+\epsilon\frac{2(\xi-b)+\kappa(2ab-1)}{6ab}+k_{22}\frac{K^2}{4}\epsilon^2+\cO(3),
\eenn
where $\cO(3)=\cO(r^3,r^2\epsilon,r\epsilon^2,\epsilon^3)$. Computing 
\benn
k_{22}\frac{K^2}{4}=\frac{\kappa(1+4ab)}{24ab}(2(b-\xi)+\kappa(1-2ab))
\eenn
yields Proposition \ref{prop:twocases}.

\section{Another Auxiliary Center Manifold Reduction}
\label{ap:cm2}

As for the center manifold reduction in Appendix \ref{ap:cm1} we present some of the 
important details for the center manifold calculation. As before for the previous appendix, 
the notation here only pertains to this calculation and should not be confused with variables 
within the main text. It is convenient to translate \eqref{eq:Olsen1_bu} via 
$B_2:=b_2-\xi$, to re-label $x_2=X_2$, $y_2=Y_2$ and change to the time scale 
$\tau=s/\epsilon^2$ which yields
\be
\label{eq:Olsen1_bu_scale_translate}
\begin{array}{rcl}
\dot{X}_2&=& 3a_2(B_2+\xi)Y_2-X_2^2+B_2X_2+\delta,\\
\dot{a}_2&=& \epsilon^2(\mu-\alpha a_2 -a_2(B_2+\xi)Y_2),\\
\dot{B}_2&=& \epsilon^2\epsilon_b(1-(B_2+\xi)X_2 -a_2(B_2+\xi)Y_2),\\
\dot{\epsilon}&=& 0,\\
\dot{\delta}&=& 0,\\
\dot{Y}_2&=& \kappa(X_2^2-Y_2-a_2(B_2+\xi)Y_2).\\
\end{array}
\ee 
The system \eqref{eq:Olsen1_bu_scale_translate} has a line of equilibrium points 
\benn
\cE_2:=\{(X_2,a_2,B_2,\epsilon,\delta,Y_2)=(0,a_2,0,0,0,0)\},
\eenn
which is degenerate since the linearization of \eqref{eq:Olsen1_bu_scale_translate} 
at $\cE_2$ for fixed $a_2=a_0$ is
\be
\label{eq:lin_chart2_fine}
A=\left.D_{(X_2,a_2,b_2,\epsilon,\delta,Y_2)}
\left(\begin{array}{c}  \dot{X}_2 \\ \dot{a}_2 \\ \dot{B}_2 \\ 
\dot{\epsilon} \\ \dot{\delta} \\ \dot{Y}_2\\ \end{array}\right)
\right|_{\cE_2}=
\left(\begin{array}{cccccc}
0 & 0 & 0 & 0 & 1 & 3a_0\xi \\
0 & 0 & 0 & \mu-\alpha a_0 & 0 & 0\\
0 & 0 & 0 & \epsilon_b & 0 & 0\\
0 & 0 & 0 & 0 & 0 & 0\\
0 & 0 & 0 & 0 & 0 & 0\\
0 & 0 & 0 & 0 & 0 & -\kappa(1+a_0\xi)\\
      \end{array}
\right).
\ee
This matrix has one negative eigenvalue $-\kappa(1+a_0\xi)$ and a 
quintuple zero eigenvalue. Hence a center manifold reduction to a five-dimensional 
center flow is required to resolve the dynamics near $\cE_2$. However, we use a 
preliminary transformation to get the system into standard form. Let $Z=(X_2,Y_2)^T$ 
and set 
\benn
A_{XY}:=\left(
\begin{array}{cc}
0 & 3a_0\xi \\
0 & -\kappa(1+a_0\xi)  \\
\end{array}
\right)\qquad \text{and}
\qquad
M:=\left(
\begin{array}{cc}
1 & -\frac{3 a_0 \xi}{\kappa (1 + a_0 \xi)}  \\
0 & 1  \\
\end{array}
\right).
\eenn
Then consider new coordinates via $M\tilde{Z}=Z$ and observe that in the 
coordinates $(\tilde{X}_2,\tilde{Y}_2)^T$ we have
\benn
\tilde{Z}'=M^{-1}A_{XY}M\tilde{Z}+\text{h.o.t.}=\left(
\begin{array}{cc}
0 & 0 \\
0 & -\kappa(1+a_0\xi)  \\
\end{array}
\right)\left(\begin{array}{c}
\tilde{X}_2\\
\tilde{Y}_2 \\
\end{array}
\right)+\text{h.o.t.},
\eenn
where $\text{h.o.t.}$ denotes higher-order terms. Let
$x:=(x_1,x_2,x_3,x_4,x_5)=(\tilde{X}_2,a_0-a_2,b_2,\epsilon,\delta)$ so
that $y=\tilde{Y}_2$ is the transformation of \eqref{eq:Olsen1_bu_scale_translate} 
into new coordinates
\benn
\begin{array}{lcl}
x_1'&=& x_5 + 3 x_3 y (x_3 + \xi) + 
x_3 \left(x_1 - \frac{3 a_0 y \xi}{\kappa + a_0 \kappa \xi} \right) 
- \left(x_1 - \frac{3 a_0 y \xi}{\kappa + a_0 \kappa \xi}\right)^2\\ 
&&- 3 a_0 \xi \frac{y + (a_0-x_2) y (x_3 + \xi) - 
\left(x_1 - \frac{3 a_0 y \xi}{\kappa + a_0 \kappa \xi}\right)^2}{1 + a_0 \xi},\\
x_2'&=& x_4 (\mu - (a_0-x_2) (\alpha + y (x_3 + \xi))),\\
x_3'&=& x_4 \epsilon_b (1 - (a_0-x_2) y (x_3 + \xi) - 
(x_3 + \xi) \left(x_1 - \frac{3 a_0 y \xi}{\kappa + a_0 \kappa \xi}\right),\\
x_4'&=&0,\\
x_5'&=&0,\\
y'&=& \kappa\left(-y - (a_0-x_2) y (x_3 + \xi) + 
\left(x_1 - \frac{3 a_0 y \xi}{\kappa + a_0 \kappa \xi}\right)^2\right),
\end{array}
\eenn
which is a vector field we denote by $(Cx+F(x,y),Py+G(x,y))^T$ for 
$F(x,y)\in\R^5$, $G(x,y)\in\R$. Observe that 
\benn
C=\{A_{ij}\}_{i,j=1}^5,\qquad \text{and}\qquad P= -\kappa(1+a_0\xi).
\eenn
The vector field is now in the correct form for center manifold theory, applied along 
the entire line of points parametrized by $a_0$. The ansatz is 
\benn
y=h(x)=\sum_{i+j=2,i\leq j}c_{ij}x_ix_j.
\eenn
The usual invariance equation is given by
\benn
Dh(x)[Cx+F(x,h(x))]=Ph(x)+G(x,h(x)),
\eenn
where different powers $x_ix_j$ have to have equal coefficients on both sides. 
This procedure yields
\benn
c_{11}=\frac{1}{1+a_0\xi},\qquad c_{15}=-\frac{1}{\kappa(1+a_0\xi)^2},
\qquad c_{55}=-\frac{1}{\kappa^2(1+a_0\xi)^3}.
\eenn 
All other coefficients $c_{ij}$ have vanish. Hence the center manifold is given 
to lowest order by
\be
\tilde{Y}_2=\frac{\tilde{X}_2^2}{1+a_0\xi}
-\frac{\tilde{X}_2\delta}{\kappa(1+a_0\xi)^2}-\frac{\delta^2}{\kappa^2(1+a_0\xi)^3}.
\ee
Transforming back to original coordinates and keeping lowest order terms yields
\be
\label{eq_cm_chart2}
Y_2=\frac{X_2^2}{1+a_0\xi}-\frac{\delta X_2}{\kappa(1+a_0\xi)^2}
+\frac{\delta^2}{\kappa^2(1+a_0\xi)^3}+\cO(Y_2^2,X_2^3,X_2Y_2,\delta Y_2,\delta^3).
\ee
Substituting the result into \eqref{eq:Olsen1_bu_scale_translate} gives, up 
to leading order, the center flow
\be
\label{eq:Olsen1_bu_scale_translate_CM}
\begin{array}{rcl}
\epsilon^2\frac{dX_2}{ds}&=& X_2^2\left(\frac{2a_0\xi-1}{1+a_0\xi}\right)
+X_2(\frac{-\delta}{\kappa(1+a_0\xi)^2}+B_2)
+\delta+\frac{\delta^2}{\kappa^2(1+a_0\xi)^3}+\cO(3),\\
\frac{da_2}{ds}&=& \mu-\alpha a_2+\cO(2),\\
\frac{dB_2}{ds}&=& \epsilon_b+\cO(2),\\
\end{array}
\ee 
which is precisely the result we wanted to prove.

\newpage

\tableofcontents

\end{document}